\documentclass[11pt]{amsart}

\usepackage{amsmath, amssymb, amscd,graphicx}

\newtheorem{thm}{Theorem}[section]
\newtheorem{claim}[thm]{Claim}
\newtheorem{conj}[thm]{Conjecture}

\newtheorem{lem}[thm]{Lemma}
\newtheorem{prop}[thm]{Proposition}
\newtheorem{q}[thm]{Question}
\newtheorem{defin}[thm]{Definition}

\def\s{{\mathfrak s}}
\def\t{{\mathfrak t}}

\def\B{{\mathcal B}}

\def\L{{\mathbb L}}

\def\Q{{\mathbb Q}}

\def\Z{{\mathbb Z}}

\def\sg1{{\sigma_1}}
\def\sg2{{\sigma_2}}

\def\Char{{\mathrm{Char}}}
\def\Hom{{\mathrm{Hom}}}
\def\spc{{\mathrm{spin^c}}}
\def\Spc{{\mathrm{Spin^c}}}

\def\cok{{\mathrm{coker}}}
\def\del{{\partial}}

\def\mod{{\textup{mod} \;}}
\def\rk{{\mathrm{rk}}}

\newcommand{\into}{\hookrightarrow}
\newcommand{\onto}{\twoheadrightarrow}

\begin{document}

\title[On closed 3-braids with unknotting number one]%
{On closed 3-braids with unknotting number one}

\author[Joshua Greene]{Joshua Greene}

\address{Department of Mathematics, Princeton University\\ Princeton, NJ 08542}

\email{jegreene@math.princeton.edu}

\thanks{Partially supported by an NSF Graduate Fellowship.}

\begin{abstract}
We prove that if an alternating 3-braid knot has unknotting number one, then there must exist an unknotting crossing in any alternating diagram of it, and we enumerate such knots.  The argument combines the obstruction to unknotting number one developed by Ozsv\'ath and Szab\'o using Heegaard Floer homology, together with one coming from Donaldson's Theorem A.
\end{abstract}
\maketitle


\section{Introduction.}

The unknotting number of a classical knot $K \subset S^3$, denoted $u(K)$, is defined to be the minimum number of crossing changes needed to obtain the unknot from some diagram of $K$.  In spite of its simple definition, this invariant is notoriously difficult to compute.  Case in point: the value $u(8_{10})=2$ was unknown until 2004 \cite{OSunknotting}!  A classical lower bound involves the knot signature: $|\sigma(K)| \leq 2 u(K)$ \cite{Msig}.  Recent developments in Heegaard Floer homology and Khovanov homology have led to noteworthy progress in estimating $u(K)$ for some interesting classes of knots.  For instance, both the knot Floer and Khovanov homology theories produce a {\em concordance invariant} which provides a lower bound on the unknotting number, and indeed on the {\em slice genus}, of a knot.  In particular, Rasmussen used the concordance invariant $s$ he defined in Khovanov homology to give a combinatorial (gauge theory free) proof of the Milnor conjecture, which implies that the unknotting number of the $(p,q)$-torus knot is $(p-1)(q-1)/2$ \cite{R}.

The current work was motivated out of interest in the following conjecture of Kohn \cite[Conjecture 12]{Kohn}.

\begin{conj}\label{conj: Kohn}

If $K$ is a knot with unknotting number one, then there exists a minimal crossing diagram of $K$ which contains an unknotting crossing.

\end{conj}

\noindent Stated in this level of generality, the conjecture seems dubious. For instance, Stoimenow \cite[Example 7.1]{Stoim} has found examples of 14-crossing knots with unknotting number one, each of which possesses a minimal diagram with {\em no} unknotting crossing (although each also possesses a minimal diagram which {\em does} contain an unknotting crossing).

However, Conjecture \ref{conj: Kohn} appears rather robust for the case of an alternating knot $K$. For one thing, the minimal diagrams for $K$ in this case are all alternating \cite{Kauffman,Malt,Thistle}.  For another, any two alternating diagrams of $K$ are related by a sequence of Tait flypes and the introduction and cancellation of nugatory crossings \cite{MenThis}, and each of these operations preserves the property of a diagram possessing an unknotting crossing.  Thus, if one alternating diagram of $K$ contains an unknotting crossing, then so does any other.  Therefore, we obtain the following derivative of Kohn's conjecture.

\begin{conj}\label{conj: main}

If $K$ is an alternating knot with unknotting number one, then any alternating diagram of $K$ must contain an unknotting crossing.

\end{conj}

Closely related to Conjecture \ref{conj: main} is an elegant result of Tsukamoto, which characterizes the alternating diagrams which contain an unknotting crossing \cite{T}.  In short, Tsukamoto's theorem provides a simple algorithm to test whether a given crossing in an alternating diagram is an unknotting crossing.  An affirmative answer to Conjecture \ref{conj: main} would therefore couple with Tsukamoto's theorem to give a simple algorithm to test whether an alternating knot has unknotting number one.


Previously, Conjecture \ref{conj: main} was known to hold for some broad classes of alternating knots: two-bridge knots \cite{KM,Kohn}, alternating large algebraic knots \cite{GLu}, and alternating knots with up to 10 crossings \cite{GLu,OSunknotting}.  Furthermore, the methodology of \cite{GLu} applies to show that Conjecture \ref{conj: main} holds for all but at most 100 11-crossing alternating knots \cite{knotinfo}.

Our main result is the validity of Conjecture \ref{conj: main} for alternating 3-braid knots.

\begin{thm}\label{thm: main}

If $K$ is an alternating 3-braid knot with unknotting number one, then any alternating diagram of $K$ must contain an unknotting crossing.

\end{thm} \noindent  Combined with Proposition \ref{p: almost-alt}, which characterizes the alternating 3-braid knot diagrams containing an unknotting crossing, Theorem \ref{thm: main} leads to the enumeration of the alternating 3-braid knots with unknotting number one.  Moreover, we argue in Proposition \ref{p: d bound} that any 3-braid knot with unknotting number 1 is ``close" to being alternating.  Furthermore, we settle Conjecture \ref{conj: main} for {\em all} 11-crossing alternating knots.


\subsection{Methodology.}  Theorem \ref{thm: main} follows by an application of Theorem \ref{t: criterion}, an algebraic-combinatorial obstruction to an alternating knot having unknotting number one.  At the heart of the method is a simple observation known as the {\em Montesinos trick}: if $u(K)=1$, then $\Sigma(K)$, the double-cover of $S^3$ branched along $K$, arises as $1/2$-integer surgery on some other knot $\kappa \subset S^3$ \cite{Monty}.  Thus, if we can obstruct $\Sigma(K)$ from arising as such a surgery, then it follows that $u(K) > 1$.  For example, if $H_1(\Sigma(K))$ fails to be a cyclic group, then $u(K)>1$ follows.

Earlier researchers have developed two finer obstructions stemming from the Montesinos trick. We briefly sketch both, and return to them in greater detail in Section \ref{s: criterion}.  First, suppose that $\Sigma(K)$ is known to bound a smooth, negative-definite 4-manifold $X$.  If $-\Sigma(K) = S^3_{-D/2}(\kappa)$ for some $D > 0$, then $-\Sigma(K)$ bounds a smooth, negative-definite 4-manifold $W$ with $b_2(W)=2$.  Gluing $X$ and $W$ along their common boundary results in a closed, smooth, negative-definite 4-manifold.  By Donaldson's Theorem A, its intersection pairing is diagonalizable \cite{D}.  This places a restriction on the intersection pairing on $X$, and by way of this restriction, Cochran and Lickorish were able to obtain some results on signed unknotting numbers \cite{CL}.  Second, suppose that $\Sigma(K)$ is a Heegaard Floer L-space with known correction terms.  If $\Sigma(K)$ is a $1/2$-integer surgery, then these values must obey a special symmetry.  By way of this method, Ozsv\'ath and Szab\'o were able to determine all the alternating knots with $\leq 10$ crossings with unknotting number one, as well as some non-alternating ones \cite{OSunknotting}.


The basic technical advance made in this work is a way to combine the obstructions stemming from Donaldson's Theorem A and the correction terms in order to develop a stronger restriction on a knot to have unknotting number one.  The way in which the two combine is reminiscent of (and indeed inspired by) a related obstruction to a knot being {\em smoothly slice} \cite{GJ}.  In short, Donaldson's Theorem A gets applied to show that a certain lattice associated to an alternating knot $K$ must embed as a sublattice of the standard $\Z^n$ lattice if $u(K) = 1$, and then the correction terms provide a sharper restriction on the embedding.  The precise statement is given in Theorem \ref{t: criterion}.  Remarkably, in the application to Theorem \ref{thm: main}, it turns out that if the lattice associated to an alternating 3-braid knot fulfills the conclusion of Theorem \ref{t: criterion}, then the embedding given therein actually identifies an unknotting crossing in the standard alternating 3-braid closure diagram.  The strength of Theorem \ref{t: criterion} is somewhat surprising, and it is natural to probe the limits of its strength.


\begin{q}\label{q: main}

Suppose that $K$ is an alternating knot with $|\sigma(K)| \leq 2$, $H_1(\Sigma(K))$ cyclic, and which fulfills the conclusion of Theorem \ref{t: criterion}.  Does it follow that an alternating diagram of $K$ must contain an unknotting crossing?

\end{q}

Clearly, an affirmative answer to Question \ref{q: main} would entail one to Conjecture \ref{conj: main} as well.  Shy of attacking Question \ref{q: main} in full, it would be interesting to pursue it in some special situations, for example alternating $n$-braid knots with unknotting number one.  Moreover, a version of Theorem \ref{t: criterion} applies to some non-alternating knots as well, and seems ripe to apply to the classification of {\em all} $3$-braid knots with unknotting number one.  However, some algebraic complications arise in the general case which have yet to be resolved at the time of this writing.  We discuss this situation in the concluding section, and hope to return to this topic in a sequel.


\subsection{Organization.}  The remainder of the paper is organized as follows.  In Section \ref{s: background}, we collect some basic notions regarding 3-braid knots and the Goeritz matrix associated to a knot diagram.  In particular, we describe the alternating 3-braid diagrams containing an unknotting crossing in Proposition \ref{p: almost-alt}, and show in Proposition \ref{p: d bound} how any 3-braid knot with unknotting number one is close to being alternating.  In Section \ref{s: HF} we provide the necessary background on intersection pairings and the correction terms from Heegaard Floer homology.   Lemma \ref{l: d diff}, especially in the equivalent form given in Lemma \ref{l: d diff 2}, is the key technical result of that section.  In Section \ref{s: criterion}, we state a precise version of the Montesinos trick, and discuss the obstructions to unknotting number one to which it leads via Donaldson's Theorem A and the Heegaard Floer correction terms.  That section culminates in the proof of Theorem \ref{t: criterion}, and illustrates it by way of a couple examples and the application to 11-crossing knots.  In Section \ref{s: sig 2}, we swiftly deduce Theorem \ref{thm: main} for the case of an alternating 3-braid knot with non-zero signature by an application of Theorem \ref{t: criterion}.  The argument there stands in marked contrast to that given in Section \ref{s: sig 0}, in which we prove Theorem \ref{thm: main} for the case of an alternating 3-braid knot with zero signature.  In that case we must delve more deeply into the combinatorics of the embedding matrix given in Theorem \ref{t: criterion} to obtain the desired conclusion.  The methodology used therein draws inspiration from work of Lisca \cite{Lisca}.  The concluding Section \ref{s: conc} discusses some directions for future work, a remark concerning quasi-alternating links, and some further justification behind Conjecture \ref{conj: main}.

To quickly navigate to the gist of our technique, we recommend first skimming the background Section \ref{s: background}, skipping over the proof of Proposition \ref{p: almost-alt} through the end of Subsection \ref{ss: 3braids}, then reading the statement of Theorem \ref{t: criterion}, and finally reading its sample applications in Subsection \ref{ss: example}.

\section*{Acknowledgments.}

It is a pleasure to thank my advisor, Zolt\'an Szab\'o, for his continued guidance and support.  Thanks in addition to Slaven Jabuka, Chuck Livingston, Jake Rasmussen, and Eric Staron for helpful correspondence, and to John Baldwin and Ina Petkova for their continued interest.


\section{3-braid knots and the Goeritz form.}\label{s: background}

\subsection{3-braid knots.}\label{ss: 3braids}

Denote the standard generators of the braid group $B_3$ by $\sigma_1$ and $\sigma_2$.  Let $h = (\sigma_1 \sigma_2)^3$.  The following result of Murasugi \cite[Proposition 2.1]{M3braid} gives a normal form for 3-braids.

\begin{prop}\label{p: normalform}

Any 3-braid is equivalent, up to conjugation, to exactly one of the form $h^d \cdot w$, where $d \in \Z$ and $w$ is either

\begin{enumerate}

\item an alternating word $\sigma_1^{-a_1} \sigma_2^{b_1} \cdots \sigma_1^{-a_m} \sigma_2^{b_m}$ with $m \geq 1$ and all $a_i, b_i \geq 1$;

\item $(\sigma_1 \sigma_2)^k$ for some $k \in \{ 0,1,2 \}$; or

\item $\sigma_1 \sigma_2 \sigma_1$, $\sigma_1^{-k}$, or $\sigma_2^k$, for some $k \geq 1$.


\end{enumerate}

\end{prop}

Observe that the closure of a 3-braid in normal form is a {\em knot} only in case (1) and in case (2) with $k \ne 0$.  We call such a knot a {\em 3-braid knot} and its {\em normal form} its realization as the closure of a 3-braid in normal form. The case $w = \sigma_1 \sigma_2 $ corresponds to the torus knot $T(3,3d+1)$, while the case $w = (\sigma_1 \sigma_2)^2$ corresponds to $T(3,3d+2)$.  A cyclic, {\em almost-alternating} 3-braid word is the result of changing some crossing in a cyclic, alternating 3-braid word.


\begin{prop}\label{p: almost-alt}

Suppose that $w$ is a cyclic, almost-alternating 3-braid word whose closure is a knot with determinant one.  Then that knot is the unknot, and $w$ takes one of the following forms:

\begin{enumerate}

\item $\sigma_1^{-n_1-1\pm1} \sigma_2^{n_2} \cdots \sigma_1^{-n_{k-1}} \sigma_2^{n_k+1} \sigma_1 \sigma_2 \sigma_1^{-n_k} \sigma_2^{n_{k-1}} \cdots \sigma_1^{-n_2} \sigma_2^{n_1-1}$, with $k \geq 2$ even and $n_1,\dots,n_k \geq 1$ (and every unlabeled exponent of the form $\pm n_i$);

\item $\sigma_1^{-n_1-1\pm1} \sigma_2^{n_2} \cdots \sigma_1^{-n_k} \sigma_2 \sigma_1 \sigma_2^{n_k+1} \cdots \sigma_1^{-n_2} \sigma_2^{n_1-1}$, with $k \geq 3$ odd and $n_1,\dots,n_k \geq 1$;

\item $\sigma_1^{-n-1\pm1} \sigma_2 \sigma_1 \sigma_2^n$, with $n \geq 0$;

\item $\sigma_1^{-2} \sigma_1 \sigma_2, \sigma_1 \sigma_1^{-2} \sigma_2, \sigma_1^{-1} \sigma_1 \sigma_1^{-1} \sigma_2$;

\item the result of swapping $\sigma_1^{-1}$ and $\sigma_2$ in one of these words; or

\item the inverse of one of these words (including the previous case).

\end{enumerate}

\end{prop}

\noindent Therefore, the list of alternating 3-braid diagrams with an unknotting crossing is gotten from the list in Proposition \ref{p: almost-alt} by replacing the single $\sigma_1$ in each of (1)-(4) by $\sigma_1^{-1}$.

\begin{proof}

Suppose that $w$ is a cyclic, almost-alternating 3-braid word, and assume that the corresponding alternating braid is a word in $\sigma_1^{-1}$ and $\sigma_2$, and that $w$ is the result of changing one $\sigma_1^{-1}$ to a $\sigma_1$.  Repeatedly make the substitutions $\sigma_2 \sigma_1 \sigma_2 \sigma_1^{-1} \to \sigma_1 \sigma_2$ and $\sigma_1^{-1} \sigma_2 \sigma_1 \sigma_2 \to \sigma_2 \sigma_1$.  When this is no longer possible, we obtain an equivalent word which (a) contains $\sigma_1 \sigma_1^{-1}$ or $\sigma_1^{-1} \sigma_1$, (b) contains $\sigma_2^2 \sigma_1 \sigma_2^2$, or (c) equals one of $\sigma_1 \sigma_2$, $\sigma_1 \sigma_2^2$, or $\sigma_1 \sigma_2^3$.  In case (a), cancel to obtain a word in $\sigma_1^{-1}$ and $\sigma_2$.  In case (b), substitute $\sigma_2^2 \sigma_1 \sigma_2^2 = \sigma_1^{-1} (\sigma_1 \sigma_2)^3$ to obtain $h$ times a word in $\sigma_1^{-1}$ and $\sigma_2$.


Now suppose that the closure of $w$ is a knot $K$ with determinant one.  In case (a), we see that $K$ is an alternating knot, so it must be the unknot.  It follows that the word we obtain must be $\sigma_1^{-1} \sigma_2$.  In case (b), the word is either $h \cdot \sigma_1^{-k}$ or $h \cdot \sigma_2^k$ for some $k \geq 0$, or $h$ times an alternating word in $\sigma_1^{-1}$ and $\sigma_2$.  However, the first two possibilities do not yield knots, and the third one yields a link of determinant $\geq 5$ \cite[Proposition 5.1]{M3braid}, so this case does not occur.  In case (c), the only possibility is $\sigma_1 \sigma_2$, which again gives the unknot.

Reversing the above procedure leads to a method to produce the almost-alternating 3-braid words representing the unknot:  begin with one of the two words $\sigma_1^{-1} \sigma_2$, $\sigma_1 \sigma_2$, in the first case insert $\sigma_1^{-1} \sigma_1$ or $\sigma_1 \sigma_1^{-1}$ someplace, and proceed by making substitutions $\sigma_1 \sigma_2 \to \sigma_2 \sigma_1 \sigma_2 \sigma_1^{-1}$ and $\sigma_2 \sigma_1 \to \sigma_1^{-1} \sigma_2 \sigma_1 \sigma_2$.  The words we get in this way, together with swapping the roles of $\sigma_1^{-1}$ and $\sigma_2$, and taking inverses of all these words, constitute the desired set.  From this description, it is straightforward to obtain the list given in the statement of the Proposition.

\end{proof}

The following result of Erle \cite{E} determines the signature of a 3-braid knot in normal form.\footnote{We adhere to the convention that a positive knot has negative signature.  The signature formula for the torus knots predates Erle's work; see \cite[Proposition 9.1]{M3braid}, which actually uses the opposite convention on signature.}

\begin{prop}\label{p: sig}

Let $K_d$ denote the $3$-braid knot with normal form $h^d \cdot w$, where $w$ is as in Proposition \ref{p: normalform}(1). Then $\sigma(K) = -4d + \sum_{i=1}^m (a_i - b_i)$.  In addition, $\sigma(T(3,6d \pm 1)) = -8d$ and $\sigma(T(3,6d \pm 2)) = -8d \mp 2$.

\end{prop}

The next Proposition determines the Rasmussen $s$-invariant of a $3$-braid knot \cite{R}.\footnote{Conjecturally, this same result applies to $2\tau$ as well, where $\tau$ denotes the Ozsv\'ath-Szab\'o concordance invariant \cite{OStau}.}  This result and the next are not needed for the proof of Theorem \ref{thm: main}, but we include them with a view towards future work, as discussed in Section \ref{s: conc}.

\begin{prop}\label{p: s}

With notation as in Proposition \ref{p: sig}, we have \begin{eqnarray*}
s(K_d) = \left\{
\begin{array}{cl}
6d-2-\sigma(K_0), & \quad \mbox{if $d > 0$}; \cr
-\sigma(K_0), & \quad \mbox{if $d = 0$}; \cr
6d+2-\sigma(K_0), & \quad \mbox{if $d < 0$}. \cr
\end{array}
\right.
\end{eqnarray*} In addition, $s(T(3,n)) = 2(n-1)$ for $n \geq 1$ and $2(n+1)$ for $n \leq -1$.

\end{prop}

\begin{proof}

When $d = 0$ or $\pm 1$, the knot $K_d$ is quasi-alternating \cite[Theorem 8.7]{B}.  Hence $s(K_d) = -\sigma(K_d)$ \cite[Theorem 1]{MOzs}\footnote{This paper takes as $s$ the quantity we call $s/2$.}, and the result follows in this case from Proposition \ref{p: sig}.  In general, \[ 6(i-j)-4 \leq s(K_i) - s(K_j) \leq 6(i-j) \] for all $i \geq j$ \cite[Theorem 9]{V}.  Notice that we let $s/2$ play the role of the function $\nu$ appearing in \cite{V}.  Assume that $d > 0$, and take $i = d$ and $j = -1$ and $1$ in this inequality.  Then \[6(i+1)-4 \leq s(K_d) - s(K_{-1}) \quad \mbox{and} \quad s(K_d) - s(K_1) \leq 6(d-1).\]  Since $s(K_{-1}) = -\sigma(K_0) - 4$ and $s(K_1) = -\sigma(K_0) + 4$, the result $s(K_d) = 6d-2-\sigma(K_0)$ follows.  The case of $d < 0$ is similar.  The assertion about the torus knot $T(3,n)$ follows from the fact that this knot is positive and so $s(T(3,n)) = 2g(T(3,n)) = 2(n-1)$ when $n \geq 1$, and the analogous fact about the mirror of $T(3,n)$ when $n \leq -1$ \cite[Theorem 4]{R}.

\end{proof}


\begin{prop}\label{p: d bound}

If $K$ is a 3-braid knot with unknotting number one and non-negative signature, then $d \in \{-1,0,1,2 \}$ when $K$ is put in normal form.

\end{prop}

\begin{proof}


Recall that if $K_-$ and $K_+$ are two knots which differ at a crossing, which is negative in $K_-$ and positive in $K_+$, then \[ 0 \leq \sigma(K_-) - \sigma(K_+) \leq 2 \quad \mbox{and} \quad -2 \leq s(K_-) - s(K_+) \leq 0 \] (\cite[Proposition 2.1]{CL}, \cite[Corollary 4.3]{R}).  In particular, the bound $|\sigma(K)| \leq 2 u(K)$ follows.  Thus, if $K$ is a torus knot, then the result follows from this bound and Proposition \ref{p: sig}. In case $K$ is not a torus knot, we suppose first that $K$ can be unknotted by changing a negative crossing to a positive one.  Then \[ 0 \leq \sigma(K) \leq 2 \quad \mbox{and} \quad -2 \leq s(K) \leq 0. \] Conditioning on the possibilities that $\sigma(K) = 0$ or $2$ and $d$ is positive or not, and applying Propositions \ref{p: sig} and \ref{p: s}, we obtain the desired bound on $d$.  The same reasoning applies if $K$ can be unknotted by changing a negative crossing.

\end{proof}









\subsection{The Goeritz form.}\label{ss: Goeritz}

Consider a diagram $D$ of a knot $K$.  It splits the plane into connected regions, which we color white and black in checkerboard fashion.  With respect to this coloration, each crossing $c$ in $D$ has an incidence number $\mu(c) = \pm 1$ as displayed in Figure \ref{f: crossingincidence}.

\begin{figure}
\centering
\includegraphics[width=3in]{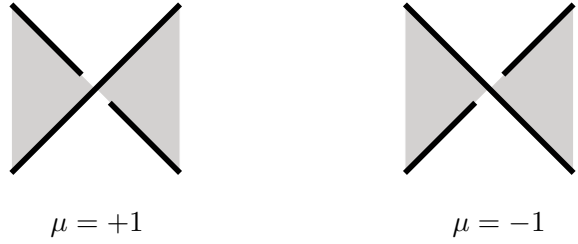}
\put(-200,-20){$\mu = +1$}
\put(-48,-20){$\mu = -1 $}
\caption{The incidence number of a crossing.}  \label{f: crossingincidence}
\end{figure}

We form a planar graph by drawing a vertex in every white region and an edge for every crossing that joins two white regions.  Associate the label $\mu(e) := \mu(c)$ to the edge $e$ corresponding to the crossing $c$, and mark a single vertex.  We refer to this decorated plane drawing $\Gamma$ as the \emph{white graph} corresponding to $D$ with the choice of marked region.

The {\em Goeritz matrix} $G = (g_{ij})$ corresponding to $\Gamma$ is defined as follows (cf. \cite[pp. 98-99]{Lick}).  Enumerate the vertices of $\Gamma$ by $v_1,\dots,v_{r+1}$, where $v_{r+1}$ denotes the marked vertex, and for $1 \leq i,j \leq r$, set \[ g_{ij} = \sum_{\small \begin{array}{c} e \text { joining } \\ v_i \text { and } v_j \end{array}} \mu(e), \quad i \ne j, \quad \quad \text{and} \quad \quad g_{ii} = -\sum_{\small \begin{array}{c} e \text{ incident } \\ v_i \text{ once} \end{array}} \mu(e).\]  Observe that we exclude loop edges in the second summation.  The matrix $G$ is a symmetric $r \times r$ matrix, so induces a quadratic form (denoted by the same symbol), and $|\det(G)| = \det(K)$.

When $D$ is an alternating diagram, it has a preferred coloration according to the convention that all crossings have incidence number $\mu = + 1$.  With this convention fixed, the Goeritz form of an alternating knot diagram is negative-definite.  The Goeritz matrix takes a particularly simple form for the case of an alternating 3-braid knot, when we mark the region which meets the braid axis.  Let $K$ denote the closure of the alternating word $\sigma_1^{-a_1} \sigma_2^{b_1} \cdots \sigma_1^{-a_m} \sigma_2^{b_m}$ with $m \geq 1$ and all $a_i, b_i \geq 1$.  The white graph $\Gamma$ consists of vertices $v_1,\dots,v_r$ in a cycle, $r := \sum_{i=1}^m b_i$, together with the marked vertex $v_{r+1}$.  The vertex $v_i$ is connected to $v_{r+1}$ by $a_l$ parallel edges if $i = 1+b_1+\cdots+b_{l-1}$, and is not adjacent to it otherwise.  Provided that $r \geq 3$, the corresponding Goeritz matrix $G = (g_{ij})$ is given by  \[ g_{ij} = \left\{
\begin{array}{cl}
-a_l-2 , & \quad \mbox{if $i=j=1+b_1+\cdots+b_{l-1}$;} \cr
-2 , & \quad \mbox{if $i=j$ is not of this form;} \cr
1 , & \quad \mbox{if $|i-j|=1$ or $r-1$;} \cr
0 , & \quad \mbox{otherwise}.
\end{array}
\right.
\]  When $r = 1$, we obtain $G = (-a_1)$, and when $r=2$, we obtain \[G = \left( \begin{matrix} -a_1-2 & 2 \\ 2 & -a_2-2 \end{matrix} \right),\] taking the bottom-right entry to be $-2$ in case the value $a_2$ is undefined.

We close with a pair of examples, to which we return in Subsection \ref{ss: example}.  The knot $8_7$ is the closure of $\sigma_1^{-4} \sigma_2 \sigma_1^{-1} \sigma_2^2$ (Figure \ref{f: 3-braid}), and $10_{79}$ is the closure of $\sigma_1^{-3} \sigma_2^2 \sigma_1^{-2} \sigma_2^3$.  The corresponding Goeritz matrices are \[G_{8_7} = \left( \begin{matrix} -6 & 1 & 1 \\ 1 & -3 & 1 \\ 1 & 1 & -2 \end{matrix} \right) \quad \text{and} \quad G_{10_{79}} = \left( \begin{matrix} -5 & 1 & 0 & 0 & 1 \\ 1 & -2 & 1 & 0 & 0 \\ 0 & 1 & -4 & 1 & 0 \\ 0 & 0 & 1 & -2 & 1 \\ 1 & 0 & 0 & 1 & -2 \end{matrix} \right).\]  Observe that if we change the crossing in the diagram of $8_7$ indicated, then the result is the unknot, and a Goeritz form for the resulting diagram is gotten by increasing the diagonal entry $g_{22} =-3$ in $G_{8_7}$ by $2$.

\begin{figure}
\begin{center}
\includegraphics[width=5in]{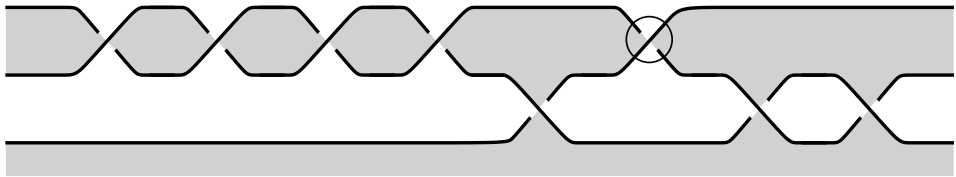}
\caption{The braid diagram representing the knot $8_7$, with an unknotting crossing indicated.  The braid axis meets the white region at the top of the diagram.} \label{f: 3-braid}
\end{center}
\end{figure}



\section{Intersection pairings and correction terms.}\label{s: HF}

Here we recall the basic facts about intersection pairings and $\spc$ structures on 4-manifolds, and the necessary input from Heegaard Floer homology.  A more extensive summary of the relevant material about the latter appears in \cite[Section 2]{OSunknotting}.  The section concludes with the statement of the versatile Lemma \ref{l: d diff 2}, which we put to use in Section \ref{s: criterion} towards the proof of Theorem \ref{t: criterion}.


\subsection{Intersection pairings and $\spc$ structures}\label{ss: int pair}

Here and throughout we take (co)homology groups with integer coefficients.  When $X$ is a compact, oriented 4-manifold with $H_2(X)$ torsion-free, there is an {\em intersection pairing} on its homology \[ \langle \cdot,\cdot \rangle : H_2(X) \otimes H_2(X) \to \Z .\]  This pairing extends to all of $H_2(X) \otimes \Q$ by linearity. The pairing is non-degenerate if and only if $\del X$ is a union of rational homology 3-spheres.  Provided that the pairing is non-degenerate and $H_1(X)$ is torsion-free, the cohomology group $H^2(X)$ is identified with the dual group $\Hom(H_2(X),\Z) \subset H_2(X) \otimes \Q$, and the pairing $\langle \cdot,\cdot \rangle$ restricts to a $\Q$-valued pairing on it.  In this way we may regard $H_2(X)$ as a subgroup of $H^2(X)$.  In topological terms, a class in $H_2(X)$ gets identified with its Poincar\'e dual in $H^2(X,\del X)$, which includes into $H^2(X)$ according to a portion of the long exact sequence in cohomology of the pair $(X,\del X)$: \[ 0 \to H^2(X, \del X) \to H^2(X) \to H^2(\del X) \to H^3(X, \del X).\]  If, moreover, $H_1(X) = 0$, then the last term in this sequence vanishes, and in this way we may identify $H^2(\del X)$ with the quotient $H^2(X) / H_2(X)$.  The {\em discriminant} of the pairing $\langle \cdot,\cdot \rangle$ is thus the order of $H^2(\del X)$.

In terms of a fixed coordinate system, let $\{ v_1,\dots,v_k \}$ denote a basis for $H_2(X)$, and express its pairing with respect to this basis by the symmetric matrix $M$.  The pairing endows the dual group $H^2(X)$ with a dual basis $\{ v_1^*,\cdots, v_k^* \}$, and the $\Q$-valued pairing on $H^2(X)$ is expressed by the inverse matrix $M^{-1}$ with respect to it.  In terms of the chosen basis, we have an identification \[ H^2(\del X) \cong \cok(M),\]  and the discriminant of the pairing is the determinant of $M$ in absolute value.

A typical way in which this setup arises is the case of a rational homology sphere $Y$ presented by surgery on an oriented, integer-framed link $\L \subset S^3$ with $k$ components and linking matrix $M$.  Let $X$ denote the trace of the surgery.  Given an oriented link component $L_i$, cap off the core of the handle attachment along it with a pushed-in oriented Seifert surface $F_i$ for $L_i$, orient the resulting surface consistently with $F_i$, and denote its class in $H_2(X)$ by $v_i$.  We obtain a basis $\{v_1,\dots,v_k\}$ for $H_2(X)$ in this way, with respect to which the intersection pairing on homology is given by the matrix $M$.

To an oriented 3-manifold $Y$ or 4-manifold $X$ we can associate its collection of $\Spc$ structures.  This set forms an affine space over the cohomology group $H^2(\cdot)$, and the {\em first Chern class} $c_1: \Spc(\cdot) \to H^2(\cdot)$ is related to the action by the formula $c_1(\s + \alpha) = c_1(\s)+ 2 \alpha$ for all $\s \in \Spc(\cdot)$ and $\alpha \in H^2(\cdot)$.  A class $\alpha \in H^2(X)$ is {\em characteristic} if \[ \langle \alpha,x \rangle \equiv \langle x, x \rangle  \; (\mod 2), \text{ for all } x \in H_2(X),\] and the set of characteristic classes is denoted $\Char(X)$.  For the case under consideration, the map $c_1: \Spc(X) \to H^2(X)$ is a 1-1 map with image $\Char(X)$.  Furthermore, when the discriminant of the pairing is odd, the map $c_1: \Spc(\del X) \to H^2(\del X)$ sets up a bijection, since 2 is a unit in $H^2(\del X)$.  The pairing in this case restricts to a non-degenerate pairing on $H_2(X) \otimes \Z / 2 \Z$.  Expressed with respect to the dual basis $\{ v_1^*,\cdots, v_k^* \}$, a characteristic class $\alpha \in H^2(X)$ is precisely one whose reduction $(\mod 2)$ agrees with that of the diagonal of $M$.

Adaptations of the preceding notions exist in the presence of torsion, but they will not be necessary here.  In summary, our working hypothesis is that $X$ is a compact 4-manifold with $H_1(X) = 0$ and $H_2(X)$ torsion-free.


\subsection{Correction terms and sharp 4-manifolds}\label{ss: sharp}

Recall that a rational homology 3-sphere $Y$ is an {\em L-space} if $HF^+_{red}(Y) = 0$, or equivalently if $\rk \; \widehat{HF}(Y) = |H_1(Y)|$.  In \cite{OSgrading}, Ozsv\'ath and Szab\'o show how to associate a numerical invariant $d(Y,\t) \in \Q$ called a {\em correction term} to an oriented, rational homology sphere $Y$ equipped with a $\spc$ structure $\t$. They prove that this invariant obeys the relation $d(-Y,\t) = - d(Y,\t)$, and if $Y$ is the boundary of a negative definite 4-manifold $X$, then \begin{equation}\label{e: d bound} c_1(\s)^2 + b_2(X) \leq 4d(Y,\t) \end{equation} for all $\s \in \Spc(X)$ for which the restriction $\s|Y$ equals $\t \in \Spc(Y)$ \cite[Theorem 9.6]{OSgrading}.

\begin{defin}\label{d: sharp} \hspace{.1in}

\begin{enumerate}

\item If $Y$ is a rational homology sphere contained in a negative-definite 4-manifold $X$, then a class $c_1(\s)$ is a {\em maximizer} if the value $c_1(\s)^2$ is maximal over all $\spc$ structures on $X$ which restrict to $\s|Y \in \Spc(Y)$.

\item  A negative definite 4-manifold $X$ with L-space boundary $Y$ is {\em sharp} if, for every $\t \in \Spc(Y)$, there is some $\s \in \Spc(X)$ with $\s|Y = \t$ that attains equality in the bound (\ref{e: d bound}).

\end{enumerate}


\end{defin}


Now suppose that $Y$ is an L-space presented by surgery on an oriented, integer-framed link $\L \subset S^3$.  Let $k$ denote the number of link components, $W$ the trace of surgery, and suppose that the linking matrix $M$ is negative-definite and has odd determinant.  Suppose that there is another oriented, framed link $\L' \subset S^3$ with the same linking matrix $M$, for which surgery on $\L'$ yields another L-space $Y'$, and for which the trace of surgery $W'$ is sharp.  Thus, we have a series of identifications \begin{equation}\label{e: spc W} \Spc(W) \overset{c_1}{\to} \Char(W) \cong \Char(W') \overset{c_1}{\gets} \Spc(W') \end{equation} and \begin{equation}\label{e: spc Y} \Spc(Y) \overset{c_1}{\to} H^2(Y) \cong \cok(M) \cong H^2(Y') \overset{c_1}{\gets} \Spc(Y').\end{equation}  Note that under the correspondence (\ref{e: spc W}), the value $c_1(\cdot)^2$ is preserved.  Suppose lastly that $-Y$ is the boundary of a negative definite 4-manifold $X$ with $H_1(X) = 0$ and $H_2(X)$ torsion-free.

\begin{lem}\label{l: d diff}

Under the stated assumptions, the restriction map $\Spc(X \cup_Y W) \to \Spc(Y)$ surjects, and using the identification $\Spc(Y) \leftrightarrow \Spc(Y')$ of (\ref{e: spc Y}), we have the inequality:

\begin{equation}\label{e: d diff}
\max_{\small \begin{array}{c} \s \in \Spc(X \cup_Y W) \\ \s|Y = \t \end{array} } c_1(\s)^2 + b_2(X \cup_Y W) \leq 4 d(Y',\t)- 4d(Y,\t).
\end{equation}  Moreover, if $X$ is sharp, then (\ref{e: d diff}) is an equality for every $\t \in \Spc(Y)$.


\end{lem}



\begin{proof}

The closed-up manifold $X \cup_Y W$ can be obtained by attaching 2-handles and a 4-handle to $X$.  Hence $H^2(X \cup W)$ is a free abelian group, $H^3(X \cup W) = 0$, and $H^1(Y) = 0$ by assumption.  Consider the Mayer-Vietoris sequence in cohomology associated to the natural decomposition of $X \cup_Y W$.  A portion of this sequence reads \[ 0 \to H^2(X \cup W) \to H^2(X) \oplus H^2(W) \to H^2(Y) \to 0. \]


\noindent Given an inclusion of a 3- or 4-manifold into a 4-manifold, the mapping $c_1$ commutes with the restriction maps on $\Spc(\cdot)$ and $H^2(\cdot)$.  Therefore, the preceding short exact sequence implies the bijection of sets \[ \begin{array}{rl} \{ \s \in \Spc(X \cup W) \; | \; \s|Y = \t \} &\overset{\sim}{\longrightarrow} \{ (\s_X, \s_W) \in \Spc(X) \times \Spc(W) \; | \; \s_X|Y = \s_W|Y = \t \} \\ \noalign{\bigskip} \s &\longmapsto (\s|X, \s|W). \end{array} \]  Since both $H_1(X)$ and $H_1(W)$ vanish, the long exact sequences in cohomology for the pairs $(X,Y)$ and $(W,Y)$ imply that the restriction maps $H^2(X) \to H^2(Y)$ and $H^2(W) \to H^2(Y)$ surject.  The subset $\Char(X) \subset H^2(X)$ is a coset of $2 H^2(X)$, which has index $2^{b_2(X)}$ in $H^2(X)$.  Since $H^2(Y)$ has odd order, it follows that the restriction map $\Spc(X) \to \Spc(Y)$ surjects as well.  The same argument applies to the map $\Spc(W) \to \Spc(Y)$, and now the above correspondence shows that the restriction map $\Spc(X \cup W) \to \Spc(Y)$ surjects, too.

Equip the free abelian groups $H^2(X \cup W)$ and $H^2(X) \oplus H^2(W)$ with their respective intersection pairings, thereby recasting the map $H^2(X \cup W) \into H^2(X) \oplus H^2(W)$ as an inclusion of (negative-definite) lattices.  With this view, the correspondence \[ c_1(\s) \mapsto (c_1(\s|X),c_1(\s|W)) \] enables us to compute \[ c_1(\s)^2 = c_1(\s|X)^2 + c_1(\s|W)^2, \] where each term is squared within its respective lattice.  By virtue of this fact, we obtain \begin{equation}\label{e: max s} \max_{\small \begin{array}{c} \s \in \Spc(X \cup W) \\ \s|Y = \t \end{array} } c_1(\s)^2 = \max_{\small \begin{array}{c} \s_X \in \Spc(X) \\ \s_X|Y = \t \end{array} } c_1(\s_X)^2 + \max_{\small \begin{array}{c} \s_W \in \Spc(W) \\ \s_W|Y = \t \end{array} } c_1(\s_W)^2. \end{equation}  In other words, a maximizer $c_1(\s)$ with $\s \in \Spc(X)$ decomposes into a pair of maximizers $(c_1(\s|X),c_1(\s|W))$.  By the correspondence (\ref{e: spc W}), we can replace the pair $(W,Y)$ appearing in the last term of (\ref{e: max s}) by the pair $(W',Y')$.  Now add the quantity $b_2(X \cup W) = b_2(X) + b_2(W)$ to both sides of this equation and invoke the inequality (\ref{e: max s}) and the sharpness hypothesis on $W'$ to obtain the inequality (\ref{e: d diff}).  The equality in case $X$ is sharp follows as well.

\end{proof}

In order to apply Lemma \ref{l: d diff}, we need to rephrase it with respect to a fixed coordinate system and invoke Donaldson's Theorem A.  To begin with, we focus on the restriction map $H^2(X \cup W) \to H^2(Y)$.  We have a sequence of inclusions of lattices which are dual to one another: \[ H_2(X) \oplus H_2(W) \into H_2(X \cup W) \cong H^2(X \cup W) \into H^2(X) \oplus H^2(W).\]  Identify the chosen basis $\{ v_1,\dots,v_k \}$ for $H_2(W)$ with its image under the first inclusion.  Given a class $\alpha \in H^2(X \cup W)$, its image under the composite $H^2(X \cup W) \into H^2(X) \oplus H^2(W) \onto H^2(W)$ is given in the dual basis $\{v_1^*,\cdots,v_k^* \}$ by $(\langle \alpha, v_1 \rangle ,\cdots, \langle \alpha,v_k \rangle)$.  Therefore, the reduction of this class in $\cok(M)$ specifies a class in $H^2(Y)$, and this is the restriction $[\alpha]$.

On the other hand, $X \cup W$ is a closed, smooth, negative-definite 4-manifold, so by Donaldson's Theorem A, the lattice $H^2(X \cup W)$ is isomorphic to $\Z^n$, $n = b_2(X \cup W)$, equipped with the standard negative definite inner product.  Choose a (negative) orthonormal basis for it.  Then the condition for a class in $H^2(X \cup W)$ to be characteristic becomes \[ \alpha \equiv {\bf 1} \; (\mod 2) \] when $\alpha$ is expressed with respect to this basis, where ${\bf 1}$ denotes the vector of all $1$'s and length $b_2(X \cup W)$.

We summarize the foregoing in the following.

\begin{lem}\label{l: d diff 2}

Under the stated assumptions, and using the identification $\Spc(Y) \leftrightarrow \cok(M) \leftrightarrow \Spc(Y')$ of (\ref{e: spc Y}), we have the inequality:

\begin{equation}\label{e: d diff 2}
\max_{\small \begin{array}{c} \alpha \equiv {\bf 1} \; (\mod 2) \\ \left[ \alpha \right] = \t \end{array} } \alpha^2 + b_2(X \cup_Y W) \leq  4 d(Y',\t)- 4d(Y,\t).
\end{equation}  Moreover, if $X$ is sharp, then (\ref{e: d diff 2}) is an equality for every $\t \in \Spc(Y)$. In any event, a maximizer $\alpha \in H^2(X \cup W)$ restricts to a maximizer $(\langle \alpha,v_1 \rangle,\dots, \langle \alpha,v_k \rangle) \in H^2(W)$.

\end{lem} \noindent The last sentence is a byproduct of the one following Equation (\ref{e: max s}). Note that $\alpha^2$ is minus the ordinary Euclidean length squared of the vector $\alpha$.  In particular, for every $\t \in \Spc(Y)$, the left-hand side of inequality (\ref{e: d diff 2}) is an even number $\leq 0$, and it equals $0$ if and only if there exists $\alpha \in \{ \pm 1 \}^n$ with $[\alpha] = \t$.

\section{A criterion for unknotting number one.}\label{s: criterion}

In this section we prove Theorem \ref{t: criterion}, which places a strong restriction on an alternating knot to have unknotting number one.  This theorem combines two earlier approaches, one due to Cochran-Lickorish using Donaldson's Theorem A, the other due to Ozsv\'ath-Szab\'o using their Heegaard Floer homology correction terms.  The two combine by way of Lemma \ref{l: d diff 2}.  Both techniques have at their core the Montesinos trick, which we state now in the precise form we need (cf. \cite[proof of Theorem 8.1]{OSunknotting}).

\begin{prop}[Signed Montesinos trick]\label{p: Mont}

Suppose that $K$ is a knot with unknotting number one, and reflect it if necessary so that it can be unknotted by changing a negative crossing to a positive one. Then $\Sigma(K) = S^3_{-\epsilon D / 2}(\kappa)$ for some knot $\kappa \subset S^3$, where $\epsilon = (-1)^{\sigma(K)/2}$, and  $D = \det(K)$.

\end{prop}


\subsection{Embeddings of intersection pairings.}\label{ss: embeddings}


For any knot $\kappa' \subset S^3$ and positive integer $D = 2n - 1$, the space $S^3_{-D/2}(\kappa')$ is the oriented boundary of the 4-manifold $W$ obtained by attaching a handle to the knot $\kappa'$ with framing $-n$ and a handle to a meridian $\mu$ of $\kappa'$ with framing $-2$.  Here $\kappa'$ and $\mu$ are regarded as knots in the boundary of a four-ball $D^4$.  Orient the knot $\kappa'$ somehow, and orient $\mu$ so the two have linking number 1.  Let $\{ x,y \}$ denote the basis for $H_2(W)$ implied by these orientations and handle attachments.  With respect to it, the intersection pairing is given by the negative definite form \begin{equation}\label{e: R_n} R_n = \left( \begin{matrix} -n & 1 \\ 1 & -2 \end{matrix} \right). \end{equation}  Now, for a knot $K$ as in the statement of Proposition \ref{p: Mont}, we must have $\sigma(K) \in \{ 0, 2 \}$.  If $\sigma(K) = 0$, then $ - \Sigma(\overline{K}) = \Sigma(K) = S^3_{-D/2}(\kappa)$, while if $\sigma(K) = 2$, then $- \Sigma(K) = S^3_{-D/2}(\overline{\kappa})$.  Here the overbar denotes mirror image.  Thus, Proposition \ref{p: Mont} leads to the following result.

\begin{prop}\label{p: W}

Assume that (i) $\sigma(K) = 0$ and $K$ can be unknotted by changing a positive crossing, or (ii) $\sigma(K)=2$ and $K$ can be unknotted by changing a negative crossing.  Then $-\Sigma(K)$ is the oriented boundary of a compact $4$-manifold $W_K$ with negative definite intersection pairing given by $R_n$.

\end{prop}

Now we specialize to the case of an alternating knot $K$.  In this case, $\Sigma(K)$ is the oriented boundary of a compact, negative-definite $4$-manifold $X_K$ with $H_2(X_K)$ torsion-free, $H_1(X_K) = 0$, and whose intersection pairing is given in a suitable basis $\{ v_1,\dots,v_r \}$ by the Goeritz matrix $G_K$ \cite[Theorem 3]{GLi}.  Consider the closed-up 4-manifold $X_K \cup_{\Sigma(K)} W_K$.  Identify the classes $v_1,\dots,v_r,x,y$ with their images under the inclusion $H_2(X_K) \oplus H_2(W_K) \into H_2(X_K \cup W_K)$.  Choose a (negative) orthonormal basis for $H_2(X_K \cup W_K)$ by Donaldson's Theorem A, and form the $(r+2) \times (r+2)$ integral matrix $A$ with row vectors $v_1,\dots,v_r,x,y$ expressed in this basis. In total, we obtain the following result.



\begin{prop}\label{p: embedding}

Suppose that $K$ is an alternating knot with unknotting number one, and without loss of generality that either (i) $\sigma(K)=0$ and $K$ can be unknotted by changing a positive crossing or (ii) $\sigma(K) = 2$.  Then there exists an $(r+2) \times (r+2)$ integer matrix $A$ for which $-A A^T = G_K \oplus R_n$.
\end{prop}

Already this result places a strong restriction on the Goeritz matrix of an alternating knot with unknotting number one.  A variant on Proposition \ref{p: embedding} appears in \cite{CL}, where it is applied to give some bounds on signed unknotting numbers.


\subsection{The correction terms test.}\label{ss: corr terms}

When $Y$ is an L-space obtained by half-integer surgery on a knot in $S^3$, Ozsv\'ath and Szab\'o prove a symmetry amongst the correction terms of $Y$ when compared with those of a corresponding lens space \cite[Theorem 4.1]{OSunknotting}.  We recall their result here.  Let $\kappa$ be a knot and $D = 2n-1$ with $n > 1$.  We have a natural identification \[ \Spc(S^3_{-D/2}(\kappa)) \to H^2(S^3_{-D/2}(\kappa)), \quad \t \mapsto c_1(\t)/2,\] since 2 is a unit in the second cohomology group.  This group is in turn isomorphic with $\cok (R_n)$, and we identify \[ \cok(R_n) \cong \Z / D \Z, \quad [(a,b)] \mapsto a + nb.\]  We note that the composite identification $H^2(S^3_{-D/2}(\kappa)) \cong \Z / D \Z$ has as its inverse the map \[\Z / D \Z \overset{\sim}{\to} H^2(S^3_{-D/2}(\kappa)), \quad i \mapsto i \cdot [x^*],\] keeping the notation of the previous Subsection.  Thus, in the case at hand, we can refine the correspondence (\ref{e: spc Y}) to \begin{equation} \Spc(S^3_{-D/2}(\kappa)) \leftrightarrow \Z / D \Z \leftrightarrow \Spc(S^3_{-D/2}(U)). \end{equation}  Under this correspondence, \cite[Theorem 4.1]{OSunknotting} reads as follows.




\begin{thm}\label{t: symmetry}

Let $\kappa$ be a knot, $D = 2n-1$ with $n > 1$, and suppose that $S^3_{-D/2}(\kappa)$ is an L-space with the property that \begin{equation}\label{e: moot} d(S^3_{-D/2}(\kappa),0) = d(S^3_{-D/2}(U),0). \end{equation}  Write $n = 2k$ or $2k+1$ depending on its parity.  Then we have the identity
\begin{equation}\label{e: symmetry}
d(S^3_{-D/2}(\kappa),i) - d(S^3_{-D/2}(U),i) = d(S^3_{-D/2}(\kappa),2k-i) - d(S^3_{-D/2}(U),2k-i)
\end{equation} for $i = 1,\dots,k$, and also for $i=0$ in case $n = 2k+1$.

\end{thm}

On the other hand, Ozsv\'ath and Szab\'o prove that $\Sigma(K)$ is an L-space when $K$ is an alternating knot, and moreover that the four-manifold $X_K$ of Subsection \ref{ss: embeddings} is sharp \cite[Proposition 3.3 and Theorem 3.4]{OSdoublecover}.  Using the equality that results in (\ref{e: max s}), this entails a formula for the correction terms of $\Sigma(K)$ in terms of the Goeritz form $G_K$ \cite[Proposition 3.2]{OSunknotting}.  This formula can be used in conjunction with Theorem \ref{t: symmetry} to prove in some cases that for a specific alternating knot $K$, the space $\Sigma(K)$ {\em cannot} be obtained by $-D/2$ surgery on any knot $\kappa$; and consequently, that the knot $K$ does not have unknotting number one.  This is the main obstruction in \cite{OSunknotting}, which was fruitfully applied there to classify alternating knots with up to ten crossings with unknotting number one, as well as to obtain results for some non-alternating examples.

In this regard, we note that $d(\Sigma(K),0) = d(S^3_{-D/2}(U),0)$ holds whenever $K$ is an alternating knot with $\sigma(K) \in \{ 0,2 \}$ and $D = \det(K)$.  That is, the hypothesis (\ref{e: moot}) is always met in the case $\Sigma(K) = S^3_{-D/2}(\kappa)$ with $K$ an alternating knot with $u(K) = 1$.  For according to \cite[Theorem 1.2]{MOw}, $d(\Sigma(K),0) = - \sigma(K)/4.$  In addition, $d(S^3_{-D/2}(U),0)$ equals $0$ if $D > 0, D \equiv 1 \; (\mod 4)$, and it equals $-1/2$ if $D > 0, D \equiv 3 \; (\mod 4)$.  This follows by calculating the square of a maximizer in $\Char(W',0)$ displayed in Table \ref{table: maximizers}.  Furthermore, by \cite[Theorem 5.6]{Msig}, $D = \det(K) \equiv \sigma(K) + 1 \; (\mod 4)$.  Thus, $d(\Sigma(K),0) = d(S^3_{-D/2}(\kappa),0) = d(S^3_{-D/2}(U),0)$, as claimed.

Note that the trace of surgery $W'$ corresponding to the space $S^3_{-D/2}(U)$ is sharp.  This can be seen, for instance, by exhibiting $S^3_{-D/2}(U)$ as the branched-double cover of the twist knot $T_n$, depicted in Figure (\ref{f: twistknot}).  By marking the outer region, we obtain the Goeritz matrix $R_n$ for this knot, and $W' \cong X_{T_n}$.  Moreover, it is a straightforward matter to identify the maximizers in $\Char(W',i) := \{ \alpha \in \Char(W') \; | \; [\alpha] = i \},$ where we use the correspondence $H^2(S^3_{-D/2}(U)) \overset{\sim}{\to} \Z / D \Z$.  They are tabulated in Table \ref{table: maximizers}.



\begin{table}[h]\vspace*{-1ex}
\caption{Maximizers $\alpha \in Char(W',i)$.}\label{table: maximizers}
\begin{tabular}{ccc}
\hline
&$n=2k$&$n=2k+1$\\
\hline
$i=0, \pm 1, \dots, \pm k$&$(2i,0)$&$(2i+1,-2),(2i-1,2)$\\
\hline
$i=\pm(k+1),\dots,\pm n$&$(2i-2n,2),(2i-2n+2,-2)$&$(2i+1-2n,0)$\\
\end{tabular}
\end{table}


\begin{figure}
\begin{center}
\includegraphics[width=2in]{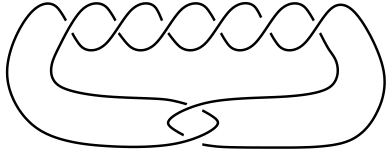}
\caption{The twist knot $T_n$, shown here for $n=6$.} \label{f: twistknot}
\end{center}
\end{figure}


\subsection{The refined test.}\label{ss: refined}

We now state Theorem \ref{t: criterion}, our primary restriction for an alternating knot $K$ to have unknotting number one.  The main ingredients in its proof are Lemma \ref{l: d diff 2} and Theorem \ref{t: symmetry}.

\begin{thm}\label{t: criterion}

Suppose that $K$ is an alternating knot with unknotting number one, and without loss of generality that either (i) $\sigma(K)=0$ and $K$ can be unknotted by changing a positive crossing or (ii) $\sigma(K) = 2$.  Let $G_K$ denote its Goeritz matrix, $\det(K) = 2n-1$, and $R_n$ the matrix in (\ref{e: R_n}). Then there exists an $(r+2) \times (r+2)$ integer matrix $A$ for which \[ -A A^T = G_K \oplus R_n \] and whose last two rows are \begin{equation}\label{e: x and y} x = (0,1,x_3,\dots,x_{r+2}) \text{ and } y = (1,-1,0,\dots,0).\end{equation}  The values $x_3,\dots,x_{r+2}$ are non-negative and obey the {\em change-making condition} \begin{equation}\label{e: change} x_3 \leq 1 \; , \; x_i \leq x_3+\cdots+x_{i-1} + 1 \text{ for } 3 < i \leq r+2,\end{equation} and the upper-right $r \times r$ submatrix $C$ of $A$ has determinant $\pm 1$.

\end{thm}

\begin{proof}

The space $\Sigma(K)$ is the oriented boundary of the sharp 4-manifold $X_K$, which arises by attaching 2-handles along a framed link $\L \subset S^3$ with linking matrix $G_K$. The signed Montesinos trick implies that $-\Sigma(K) = S^3_{-D/2}(\kappa)$, which is the boundary of the manifold $W_K$ obtained by attaching 2-handles along the framed link $\kappa \cup \mu$ with linking matrix $R_n$ (see Equation (\ref{e: R_n})).  As noted above, the manifold $W'$ corresponding to the L-space $S^3_{-D/2}(U)$ is sharp.  Therefore, the hypotheses preceding Lemma \ref{l: d diff} are fulfilled.  Furthermore, the technical hypothesis \ref{e: moot} of Theorem \ref{t: symmetry} is met as well, as remarked in Subsection \ref{ss: corr terms}.

We use the coordinate-dependent Lemma \ref{l: d diff 2} in tandem with Theorem \ref{t: symmetry} to obtain sharper information on the embedding matrix $A$ guaranteed by Proposition \ref{p: embedding}.  As before, we identify $H_2(X_K \cup W_K)$ with the lattice $\Z^{r+2}$, equipped with the standard negative definite inner product and a (negative) orthonormal basis, and label the rows of $A$ by $v_1,\dots,v_r,x,y$.  Since $y^2 = -2$, we can perform an automorphism of $\Z^{r+2}$ to arrange that $y = (1,-1,0,\dots,0)$.  Writing $x = (x_1,\dots,x_{r+2})$, the equation $\langle x,y \rangle = 1$ implies that \begin{equation}\label{e: x prelim} x_1 - x_2 = -1. \end{equation}



Suppose that $n = 2k$ is even, and consider the set \begin{equation} S := \{ 0 \leq 2j \leq 2k \; | \; d(\Sigma(K),j) = d(S^3_{-D/2}(U),j) \}.\end{equation}  Select any $j \in S$.  By Lemma \ref{l: d diff 2}, there exists a maximizer $\alpha = (\alpha_1,\dots,\alpha_{r+2}) \in \{\pm 1 \}^{r+2}$ with $[( \langle \alpha,x \rangle, \langle \alpha,y \rangle )] = j \in \Z / D \Z \cong \cok(R_n)$, and the class $( \langle \alpha,x \rangle, \langle \alpha,y \rangle) \in H^2(W_K)$ is itself a maximizer.  Referring to Table \ref{table: maximizers}, we identify this class as $(2j,0)$: that is, $ \langle \alpha,x \rangle = -2j$ and $\langle \alpha,y \rangle = 0$.  From $\langle \alpha,y \rangle = 0$ we obtain $\alpha_1 = \alpha_2 = \pm 1$.  Therefore, the expression for $\langle \alpha,x \rangle$ becomes \[ 2j = - \alpha_1 (x_1 + x_2) - \sum_{i=3}^{r+2} \alpha_i x_i. \] Conversely, any value of this form with all $\alpha_i \in \{ \pm 1 \}$ belongs to $S$.  Thus, \[ S = \{ \alpha_1 (x_1+x_2) + \sum_{i=3}^{r+2} \alpha_i x_i \geq 0 \; | \; \alpha_i \in \{ \pm 1 \} \; \forall i \}. \]  In particular, the largest element of $S$ is $S_{max} : = |x_1+x_2| + \sum_{i=3}^{r+2} |x_i|$, which is gotten by taking $\alpha_1$ to have the same sign as $(x_1 + x_2)$, and $\alpha_i$ to have the same sign as $x_i$ for $i \geq 3$.

Similarly, define \begin{equation}\label{e: S'} S' := \{ 0 \leq 2j' \leq 2k \; | \; d(\Sigma(K),2k-j') = d(S^3_{-D/2}(U),2k-j') \}.\end{equation}  The same argument, making use of Equation (\ref{e: x prelim}), identifies \[ S' = \{ 1 + \sum_{i=3}^{r+2} \alpha_i x_i \geq 0 \; | \; \alpha_i \in \{ \pm 1 \} \; \forall i \}. \]  In particular, its largest element is $S'_{max} := 1 + \sum_{i=3}^{r+2} |x_i|$.  Now, Theorem \ref{t: symmetry} implies that in fact $S = S'$.  In particular, $S_{max} = S'_{max}$.  Comparing these values, and again making use of Equation (\ref{e: x prelim}), we obtain $(x_1,x_2) = (0,1)$ or $(-1,0)$.  Replacing $x$ by $-(x+y)$ if necessary, we obtain an embedding matrix $A$ for which $(x_1,x_2) = (0,1)$.  This establishes (\ref{e: x and y}).

Let us probe the equality $S = S'$ in light of the fact that $x_1+x_2=1$.  Choose a positive value $2j \in S'$.  Then there exist $\alpha_3,\dots,\alpha_{r+2} \in \{ \pm 1 \}$ so that $1 + \sum_{i=1}^{r+2} \alpha_i x_i = 2j$.  Now choose $\alpha_1 = -1$, so that $\alpha_1 + \sum_{i=1}^{r+2} \alpha_i x_i = 2j-2 \in S$.  Thus, for every positive value $2j \in S = S'$, the value $2j-2$ belongs to $S$ as well. Hence $S$ consists of all even numbers between $0$ and $S_{max}$.  (An alternative argument proceeds by way of \cite[Theorem 8.4]{OSunknotting}.)  Thus, \[ S \cup (-S) = \{ \alpha_1 + \sum_{i=3}^{r+2} \alpha_i x_i \; | \; \alpha_i \in \{ \pm 1 \} \; \forall i \}\] consists of all even numbers between $-S_{max}$ and $S_{max}$.  Replace the $i^{th}$ basis vector by its negative if necessary so that $x_i \geq 0$, and write $\alpha_i = -1 + 2 \beta_i$, with $\beta_i \in \{ 0 , 1 \}$.  It follows that the set of values \[ \{ \sum_{i=3}^{r+2} \beta_i x_i \; | \; \beta_i \in \{ 0, 1 \} \} \] consists of all the integers between $0$ and $\sum_{i=3}^{r+2} x_i$.  That is, using coins with whole values $x_3,\dots,x_{r+2}$, it is possible to make change in any whole amount from $0$ up to the maximum possible $\sum_{i=3}^{r+2} x_i$.  Reorder the basis elements so that $x_3 \leq x_4 \leq \cdots \leq x_{r+2}$.  Then it is easy to see that this change-making condition is satisfied if and only if (\ref{e: change}) holds.

The condition on the determinant of $C$ is an algebraic consequence of the fact that $x_1 = 0, x_2 = 1$.  Specifically, let $z$ denote the vector consisting of the first $r$ values from the first column of $A$ and $\overline{x}$ the vector $(x_3,\dots,x_r)$.  Since $\langle v_i,y \rangle = 0$ for all $i$, $z$ agrees with the vector consisting of the first $r$ values from the second column of $A$.  The facts that $x_1 = 0$, $x_2 = 1$, and $\langle x,v_i \rangle = 0$ for $i = 1, \dots, r$ together imply that $C \overline{x} = -z$.  Now subtract the second column of $A$ from its first, and add $x_i$ copies of the $i^{th}$ column of the second one, producing a matrix $A'$ with $\det(A') = \det(A)$.  The upper-left $r \times 2$ submatrix of $A'$ consists entirely of $0$'s, the upper-right $r \times r$ submatrix is $C$, and the lower-left $2 \times 2$ submatrix is \[ \left( \begin{matrix} -1 & - x^2 \\ 2 & -1 \end{matrix} \right), \] which has determinant $1-2n = -D$.  It follows that $\det(A) = \det(A') = \pm D \cdot \det(C)$, and since $-D^2 = \det(G \oplus R_n) = - \det(A)^2$, we obtain $\det(C) = \pm 1$, as stated.

The preceding argument goes through with minimal change in the case $n = 2k+1$, completing the proof of the theorem.

\end{proof}


\subsection{Examples.}\label{ss: example}

Theorem \ref{t: criterion} is our main criterion for an alternating knot with unknotting number one.  Using it, we will prove Theorem \ref{thm: main} over the course of the next two sections.  As a warm-up, we apply it to the pair of examples from the end of Subsection \ref{ss: Goeritz}, and discuss the case of 11-crossing knots.

First, consider the knot $8_7$.  This is an alternating 3-braid knot with $\sigma(8_7) = 2$, $\det(8_7) = 23 = 2 \cdot 12 - 1$, and $u(8_7) = 1$.  The matrix $A$ guaranteed by Theorem \ref{t: criterion} is essentially unique in this case:

\[
A =
\left(
\begin{matrix}
0 &  0 &  1 &  2 & -1 \\
1 &  1 & -1 &  0 &  0 \\
0 &  0 &  1 & -1 &  0 \\
\hline
0 &  1 &  1 &  1 &  3 \\
1 & -1 &  0 &  0 &  0
\end{matrix}
\right).
\]  Letting $C$ denote the upper-right $3 \times 3$ submatrix, observe that the matrix $- C C^T$ is a Goeritz matrix for the knot diagram obtained on changing the crossing indicated in Figure \ref{f: 3-braid}.  Indeed, this is typical of the case of an alternating 3-braid knot with non-zero signature, and for which there is a matrix $A$ fulfilling the conclusion of Theorem \ref{t: criterion}: the matrix $-C C^T$ is the Goeritz matrix for a knot diagram obtained on changing some crossing in the given diagram.  Moreover, the resulting knot is almost-alternating, and has determinant $|\det(-C C^T)| = 1$: thus, it is the unknot, according to Proposition \ref{p: almost-alt}.  In this way, Theorem \ref{t: criterion} enables us to identify an unknotting crossing in the given knot diagram.  This is the spirit of the argument given in Section \ref{s: sig 2}.

Next, consider the knot $10_{79}$.  This is an alternating 3-braid knot with $\sigma(10_{79}) = 0$ and $\det(10_{79}) = 61 = 2 \cdot 31 -1$.  Putting aside the change-making condition (\ref{e: change}), there is an essentially unique matrix $A$ which fulfills the other conclusions of Theorem \ref{t: criterion}:

\[
A =
\left(
\begin{matrix}
 1 &  1 &  0 &  1 &  1 &  0 & -1 \\
 0 &  0 &  0 &  0 &  0 & -1 &  1 \\
 1 &  1 &  0 &  0 &  1 &  0 & -1 \\
 0 &  0 &  0 &  1 & -1 &  0 &  0 \\
 0 &  0 &  1 & -1 &  0 &  0 &  0 \\
\hline
 0 &  1 &  2 &  2 &  2 &  3 &  3 \\
 1 & -1 &  0 &  0 &  0 &  0 &  0
\end{matrix}
\right).
\]  However, since the penultimate row fails to satisfy (\ref{e: change}), it follows that $u(10_{79}) \ne 1$.  Notice that there are two rows amongst the first $r$ in $A$ with non-zero entries in the first two columns.  This is generally the case for the matrix $A$ corresponding to an alternating 3-braid knot with zero signature, and which fulfills all the conclusions of Theorem \ref{t: criterion} except possibly the change-making condition.  We can argue further that when the change-making condition is met, those two rows are adjacent: $\langle v_i, v_j \rangle = 1$.  Granted this, we can proceed as sketched above in the case of non-zero signature to identify an unknotting crossing in the given diagram.  This is the spirit of the argument given in Section \ref{s: sig 0}.

Using Theorem \ref{t: criterion} as demonstrated, we can complete the determination of the alternating knots with unknotting number one and crossing number at most $11$.  As mentioned in the Introduction, the determination up to unknotting number $10$ follows from classical techniques, together with the work of Ozsv\'ath-Szab\'o \cite{OSunknotting} and Gordon-Luecke \cite{GLu}.  Furthermore, Gordon-Luecke succeeded in determining the 11-crossing alternating knots with unknotting number one with 100 exceptions.  The exceptions are the knots $11aN$, where $N \in \{ \underline{1}$, $\underline{4}$, $\underline{5}$, $\underline{6}$, $\underline{{\bf 7}}$, $\underline{16}$, $\underline{21}$, $\underline{23}$, $\underline{32}$, $\underline{{\bf 33}}$, $\underline{36}$, $\underline{37}$, $\underline{39}$, $\underline{42}$, $\underline{{\bf 45}}$, $\underline{46}$, $\underline{50}$, $\underline{51}$, $\underline{55}$, $\underline{58}$, $\underline{61}$, $\underline{87}$, $\underline{92}$, $\underline{97}$, {\bf 99}, $\underline{103}$, $\underline{{\bf 107}}$, $\underline{108}$, $\underline{109}$, 112, $\underline{118}$, 125, $\underline{128}$, 131, $\underline{133}$, $\underline{134}$, 135, {\bf 137}, {\bf 148}, $\underline{153}$, 155, 158, 162, {\bf 163}, 164, $\underline{165}$, 169, {\bf 170}, 171, 172, $\underline{181}$, 196, 197, $\underline{199}$, $\underline{201}$, $\underline{{\bf 202}}$, $\underline{214}$, 217, 218, $\underline{{\bf 219}}$, $\underline{221}$, {\bf 228}, 239, 248, 249, $\underline{258}$, 268, 269, 270, 271, 273, {\bf 274}, 277, 278, 279, {\bf 281}, 284, 285, 286, {\bf 288}, {\bf 296}, {\bf 297}, 301, 303, 305, 312, 313, {\bf 314}, 315, 317, 322, 324, 325, 327, 331, 332, 349, 350, 352, $\underline{362} \}$ \cite{knotinfo}.  A laborious, week-long hand calculation using Theorem \ref{t: criterion} rules out these remaining 100 possibilities.  Only four of these values, $N \in \{$55, 87, 153, 172$\}$, require the invocation of the change-making condition (\ref{e: change}).  Admittedly, this calculation is difficult to check.  Fortunately, Slaven Jabuka and Eric Staron have independently verified several of these cases by different methods; those values checked by Jabuka appear above in bold, and those by Staron appear underlined \cite{Jabuka,Staron}.  Furthermore, the knot $11a362$ is the pretzel knot $P(5,5,3)$, which has unknotting number $> 1$ by a result of Kobayashi \cite{Kobayashi}.


\section{Alternating $3$-braid closures with signature 2.}\label{s: sig 2}

Suppose that $K$ is the closure of an alternating $3$-braid $\sigma_1^{-a_1} \sigma_2^{b_1} \cdots \sigma_1^{-a_m} \sigma_2^{b_m}$ with $m \geq 1$ and all $a_i, b_i \geq 1$.  We assume that $K$ has unknotting number one and $\sigma(K) = 2$.  According to Proposition \ref{p: sig},
\begin{equation}\label{e: r sig 2} r := \sum_{i=1}^m b_i = \sum_{i=1}^m a_i - 2.
\end{equation}  By \cite[Theorem 5.6]{Msig}, $D = \det(K) \equiv \sigma(K) + 1 \; (\mod 4)$, so we write $D = 2n-1$ with $n$ even.

Now express $G_K \oplus R_n = -A A^T$ according to Theorem \ref{t: criterion}.  By abuse of notation, we identify the vertices $v_i$ of the white graph $\Gamma$ with the corresponding rows of $A$.  Thus, we write $v_i = (v_{i1},\dots,v_{i(r+2)})$.  Set
\begin{equation} v  = v_1 + \cdots + v_r
\end{equation} and observe that
\begin{equation}\label{e: vector sum} v^2 =  2r+\sum_{i=1}^m -2(b_i - 1) + \sum_{i=1}^m (-a_i-2) = 2r - \sum_{i=1}^m (a_i + 2b_i) = -(r+2), \end{equation} making use of Equation (\ref{e: r sig 2}) in the last step.  On the other hand, $v$ is characteristic on the sublattice $H_2(X_K) \oplus H_2(W_K) \subset \Z^{r+2}$, noting in particular that $\langle v,x \rangle  = 0 \equiv -n = x^2 \; (\mod 2)$.  Since $H_2(X_K) \oplus H_2(W_K)$ has odd index in $\Z^{r+2}$, it follows that $v$ is characteristic for $\Z^{r+2}$.  Hence $v \equiv {\bf 1} \; (\mod 2)$, and by Equation (\ref{e: vector sum}) we conclude that $v \in \{ \pm 1 \}^{r+2}$.  Moreover, by replacing each basis vector by its negative as necessary, we may assume that $v =  {\bf 1}$.  Note that in so doing, some of the components of $x$ and $y$ as stated in Theorem \ref{t: criterion} may become negated.  However, it is still the case that the first two entries of $y$ are negatives of one another: this is because $\langle v_i,y \rangle = 0$ for all $i$, and so $0 = \sum_{i=1}^r \langle v_i,y \rangle = \langle v,y \rangle = \langle {\bf 1},y \rangle = -(y_1 + y_2)$.  Again invoking the fact that $\langle v_i,y \rangle = 0$, we learn that \begin{equation}\label{e: vi1 vi2} v_{i1} = v_{i2} \text{ for all } i. \end{equation}

Next, fix an index $i$.  Then \[ \sum_j v_{ij} = - \langle v,v_i \rangle = -(v_i^2 + 2) = \sum_j v_{ij}^2 - 2.\]  Since the $v_{ij}$ are integers, it must be the case that
\begin{equation}\label{e: vij} v_{ij} = 0 \mbox{ or } 1 \mbox{ for all but a single index } j, \mbox{ for which } v_{ij} = -1 \mbox{ or } 2. \end{equation}  Since $\sum_{i=1}^r v_{i1} = 1$, it follows that there is some index $i$ for which $v_{i1} \ne 0$.  Now (\ref{e: vi1 vi2}) and (\ref{e: vij}) together imply that $v_{i1} = v_{i2} = 1$.  Moreover, we see that the index $i$ is unique, and so $v_{j1} = v_{j2} = 0$ for all $j \ne i$.  In addition, $v_i$ must have some non-zero coordinate besides $v_{i1}$ and $v_{i2}$, since $\langle v_i,v_{i+1} \rangle = 1$.   It follows that $v_i^2 < -2$.  Consequently, the corresponding vertex $v_i$ in the white graph $\Gamma$ has some edge to the marked vertex $v_{r+1}$.  By changing the crossing in the knot diagram corresponding to this edge, the result is an almost-alternating $3$-braid closure $K'$ whose Goeritz matrix is given by $G' = - C C^T$.  It follows that $\det(K') = |\det(G')| = 1$, and now Proposition \ref{p: almost-alt} implies that $K' = U$.  Thus, the embedding matrix $A$ guaranteed by Theorem \ref{t: criterion} identifies an unknotting crossing in the given alternating 3-braid diagram, namely the one between $v_i$ and $v_{r+1}$.  This establishes Theorem \ref{thm: main} for the case of an alternating 3-braid knot with non-zero signature.

\section{Alternating $3$-braid closures with signature 0.}\label{s: sig 0}



Suppose that $K$ is the closure of an alternating $3$-braid with unknotting number one, notated as in Section \ref{s: sig 2}, and assume that $\sigma(K) = 0$.  We write $D = 2n-1$ with $n$ odd.   According to Proposition \ref{p: sig}, \begin{equation} r := \sum_{i=1}^m b_i = \sum_{i=1}^m a_i. \end{equation}

Express $G_K \oplus R_n = -A A^T$ according to Theorem \ref{t: criterion}.  We proceed as in Section \ref{s: sig 2}.  Here, however, we set \[ v = v_1 + \cdots v_r + y \] and observe that \[ v^2 =  2r+\sum_{i=1}^m -2(a_i - 1) + \sum_{i=1}^m (-b_i-2) + y^2 = 2r -2 - \sum_{i=1}^m (2a_i + b_i) = -(r+2), \] as before.  Moreover, we deduce that $v$ is characteristic, noting in particular that $\langle v,x \rangle = 1 \equiv -n = x^2 \; (\mod 2)$, and conclude in the same way as before that $v \in \{ \pm 1 \}^{r+2}$.  Replace each basis vector by its negative as necessary so that $v = {\bf 1}$, possibly altering $x$ and $y$ from the precise form stated in Theorem \ref{t: criterion}.  Indeed, it is now the case that $y_1 = y_2 = 1$, since $-y_1-y_2 = \langle {\bf 1}, y \rangle = \langle v, y \rangle = -2$.  This implies at once that $v_{i1} = -v_{i2}$ for all $i$, and that $\sum_i v_{i1} = \sum_i v_{i2} = 0$.  Furthermore, (\ref{e: vij}) holds here as well.  If $v_{i1} = v_{i2} = 0$ for all $i$, then $G = -C C^T$ and so $\det(K) = |\det(G)| = 1$.  Since $K$ is alternating, this implies that $K$ is the unknot, which is a contradiction because this knot has unknotting number {\em zero}.  Consequently, there is some index $i$ for which $v_{i1} \ne 0$.  Since $v_{i1} = - v_{i2}$, it follows from (\ref{e: vij}) that $v_{i1} = \pm 1$.  Moreover, since $\sum_i v_{i1} = 0$, there is another index $j$ for which $v_{j1} \ne 0$, and which we may therefore take to have the opposite sign as $v_{i1}$.  Now suppose that there were some third index $k$ for which $v_{k1} \ne 0$.  Then $v_k$ agrees in its first two coordinates with one of $v_i$ and $v_j$, which we may take to be $v_i$, without loss of generality.  Appealing to (\ref{e: vij}) again, the other coordinates of $v_i$ and $v_k$ are all non-negative, and so $0 \leq \langle v_i, v_k \rangle \leq -(v_{i1} v_{k1} + v_{i2} v_{k2}) = -2$, a contradiction.  In total, we have obtained the following result.

\begin{lem}\label{l: v_i & v_j}

In the matrix $A$ guaranteed by Theorem \ref{t: criterion}, we can negate some of its columns so that $v := v_1 + \cdots + v_r + y = {\bf 1}$.  In so doing, $y$ takes the form $(1,1,0,\dots,0)$, and there exist a pair of indices $i$ and $j$ for which $v_{i1} = -v_{i2} = -v_{j1} = v_{j2} = 1$, and $v_{k1} = v_{k2} = 0$ for all $k \ne i,j$.

\end{lem}

The vectors $v_i$ and $v_j$ appearing in Lemma \ref{l: v_i & v_j} play the same role in the present situation as the distinguished vector $v_i$ did in Section \ref{s: sig 2}.  The remainder of the proof of Theorem \ref{thm: main} in the case of signature $0$ reduces to establishing the following claim.

\begin{claim}\label{claim: (v_i,v_j) = 1}

The vectors $v_i$ and $v_j$ guaranteed by Lemma \ref{l: v_i & v_j} obey $\langle v_i, v_j \rangle \ne 0$.

\end{claim}

\noindent Thus, when $r > 2$, Claim \ref{claim: (v_i,v_j) = 1} amounts to the assertion that $\langle v_i, v_j \rangle = 1$.  To see how Theorem \ref{thm: main} follows, suppose that Claim \ref{claim: (v_i,v_j) = 1} holds.  Then the regions corresponding to $v_i$ and $v_j$ abut at some crossing.  Change it.  The result is a new 3-braid knot $K'$ with Goeritz matrix $G' = -C C^T$.  This is an almost-alternating 3-braid knot with determinant 1, so we conclude once more by Proposition \ref{p: almost-alt} that $K' = U$.  Hence Theorem \ref{t: criterion} again identifies an unknotting crossing in the 3-braid diagram of $K$, this time one adjoining the regions $v_i$ and $v_j$.

Curiously, Claim \ref{claim: (v_i,v_j) = 1} appears to require a substantial effort to establish, and we prove it over the course of the next few Subsections.  It would be very satisfying to obtain a proof of Theorem \ref{thm: main} in the case of signature 0 which is nearly as simple as the case of signature 2.  However, more effort is definitely needed in this case, since the change-making condition (\ref{e: change}) did not come to bear in section \ref{s: sig 2}, but it does in the present situation; compare the example of $10_{79}$ in Subsection \ref{ss: example}.

Here is an overview of our approach.  We want to study a matrix $A$ which fulfills the conclusions of Theorem \ref{t: criterion}.  To do so, we first focus on the submatrix $B$ of $A$ spanned by the rows $v_1,\dots,v_r,$ and $y$.  More precisely, we make the following definition.

\begin{defin}\label{d: B}

Given positive integers $a_1,\dots,a_m,b_1,\dots,b_m$, let $K$ denote the closure of $x^{a_1} y^{-b_1} \cdots x^{a_m} y^{-b_m}$, and let $G_K$ denote its Goeritz matrix.  Let $\B$ denote the set of those $(r+1) \times (r+2)$ integer matrices $B$ for which $- B B^T = G_K \oplus (-2)$ for some such $G_K$ and whose rows sum to {\bf 1}, and $\B_0 \subset \B$ those for which $\sum_i a_i = \sum_i b_i$.

\end{defin}

\noindent In Subsection \ref{ss: B_0}, we describe how any matrix $B \in \B_0$ can be built up by a simple process from one of three small matrices.  This is accomplished in Lemma \ref{l: expand}.  In Subsection \ref{ss: B_0 2}, we sharpen this construction in the case of a matrix $B \in \B_0$ for which $\langle v_i, v_j \rangle = 0$, which we describe precisely in Lemma \ref{l: expand 2} and refine somewhat in Lemma \ref{l: C structure}.  Finally, in Subsection \ref{ss: killer}, we show that no such matrix $B$ can be extended to a matrix $A$ fulfilling the conclusions of Theorem \ref{t: criterion}.  This establishes Claim \ref{claim: (v_i,v_j) = 1} and hence Theorem \ref{t: criterion}.



\subsection{Contraction, expansion, and the set $\B_0$.}\label{ss: B_0}
Drawing inspiration from Lisca's work \cite{Lisca}, we characterize the matrices in the set $\B_0$ by means of the process of {\em expansion}. Before introducing this notion, we state a preparatory Lemma.

\begin{lem}\label{l: B columns}

Given $B \in \B_0$, the multi-set of non-zero values in any column takes the form $\{1,1,-1\}$, $\{2,-1\}$, or $\{1\}$.

\end{lem}

\begin{proof}

We proceed in two steps, relying in each on (\ref{e: vij}).

{\em 1. If a column contains a $-1$, then it contains a single $-1$.}  For suppose that $v_s$ and $v_t$ were two distinct rows containing a $-1$ in the same column.  Every other entry in these rows is non-negative, which implies that $\langle v_s, v_t \rangle \leq -1$, a contradiction.

{\em 2. If a column contains a $2$, then it contains a $-1$, and every other entry is $0$.}  If a column contains a $2$, then for the column sum to equal $1$, it must contain some negative entry, and the only possibility is a $-1$.  If there were some additional non-zero entry in the column, then to keep the column sum $1$, there must again be another $-1$ entry, in contradiction to the first step of the argument.

The statement of the Lemma now follows.

\end{proof}

Now choose a row $v_s$ of $B$, and suppose that $v_s^2 = -2$.  If $v_s = v_i$ or $v_j$, then $\langle v_i, v_j \rangle = 2$ and $r=2$.  We handle this case separately in a moment, so for now assume that $r > 2$.  Now the entries of $v_s$ consist of one 1, one -1, and the rest 0's.  Consider the submatrix of $B$ induced on the columns containing the support of $v_s$ and on the rows whose support meets these two columns.  In light of Lemma \ref{l: B columns}, there are three possibilities.  In each case, the row induced by $v_s$ is the one containing both $1$ and $-1$.

\[ \left( \begin{matrix} 0 & 1 \\ 1 & -1 \\ 0 & 1 \end{matrix} \right), \quad \left( \begin{matrix} 1 & 2 \\ 1 & -1 \\ -1 & 0 \end{matrix} \right), \quad \left( \begin{matrix} 1 & 1 \\ 0 & 1 \\ 1 & -1 \\ -1 & 0 \end{matrix} \right). \]  Let $a$ (resp. $b$) denote in each case the row of $B$ which induces the one appearing directly above (resp. below) the one induced by $v_s$.  Let $c$ denote the remaining row in the third case.  Note that in the first case, not both of $a$ and $b$ can have square $-2$, as $\langle a, b \rangle \geq 0$.  So we assume without loss of generality that $a^2 < -2$ in that case.

\begin{defin}\label{d: contract}

By {\em contraction} we mean the following process.  Modify the vectors $a,b$, and in the third case $c$ as well, so that they induce one of the following patterns on the column containing the $-1$ entry of $v_s$:

\[ \left( \begin{matrix} 1 \\ 0 \end{matrix} \right), \quad \left( \begin{matrix} 2 \\ -1  \end{matrix} \right), \quad \left( \begin{matrix} 1 \\ 1 \\ -1 \end{matrix} \right). \]  Then delete the row $v_s$ and the column containing its $1$ entry.

\end{defin} \noindent Note that Lisca's definition of contraction \cite[Definition 3.4]{Lisca} coincides with a contraction of the first type in Definition \ref{d: contract}.  Indeed, this type will be the one we work with most often.

Let $a',b',c'$ denote the modified vectors.  Observe that $\langle a', b' \rangle = \langle a, b \rangle - 1$ following contraction, and that the pairing between any other pair of distinct vectors remains unchanged.  Furthermore, if $r = 3$, then $\langle a', b' \rangle = 2$, and if $r > 3$, then $\langle a', b' \rangle = 1$.  It follows that contraction carries a matrix $B \in \B$ to another matrix in $\B$ of smaller rank.  Moreover, if we denote by $\{ a'_i, b'_i \}$ the parameters corresponding to the contracted matrix, then $\sum_i a'_i = \sum_i b'_i$.  Therefore, the process of contraction carries a matrix $B \in \B_0$ to smaller one in $\B_0$.  It follows that by applying a sequence of contractions to any $B \in \B_0$, we obtain a matrix in $\B_0$ for which $r = 2$, or $r > 2$ and every $v_k$ has square $< -2$.  In the first case, we obtain the two possibilities

\begin{equation}\label{e: possibility 1 & 2} M_1 = \left( \begin{matrix} 1 & -1 & 1 & 1 \\ -1 & 1 & 0 & 0 \\ \hline 1 & 1 & 0 & 0 \end{matrix} \right), \quad M_2 = \left( \begin{matrix} 1 & -1 & 1 & 0 \\ -1 & 1 & 0 & 1 \\ \hline 1 & 1 & 0 & 0 \end{matrix} \right). \end{equation}  In the second case, we have $b_i = 1$ for all $i$, and so $a_i = 1$ for all $i$ as well.  Hence every row except the last has square exactly $-3$.  To satisfy $\langle v_i, v_j \rangle \geq 0$, it must be that $\langle v_i, v_j \rangle = 1$, and so $v_{i3} = v_{j3} = 1$ on permuting the columns.  Hence there is an index $k$ for which $v_{k3} = -1$.  It follows that $\langle v_i, v_k \rangle = \langle v_j, v_k \rangle = 1$, and we obtain the single possibility \begin{equation}\label{e: possibility 3} M_3 = \left( \begin{matrix} 0 & 0 & -1 & 1 & 1 \\ 1 & -1 & 1 & 0 & 0 \\ -1 & 1 & 1 & 0 & 0 \\ \hline 1 & 1 & 0 & 0 & 0 \end{matrix} \right). \end{equation}  Since contraction preserves the presence of a 2 in a matrix, and none of the matrices $M_1, M_2, M_3$ of (\ref{e: possibility 1 & 2}) and (\ref{e: possibility 3}) has one, it follows that only the first and third type of contraction occur in this process.  Before summarizing the foregoing, we make one further definition.

\begin{defin}\label{d: expand}

By {\em expansion} we mean the reverse process to contraction.  More precisely, identify vectors $a',b',$ and in the third case $c'$, for which $\langle a', b' \rangle \geq 1$, there is a column distinct from the first or second whose support is contained amongst these rows, and the submatrix induced on these rows and this column takes one of the forms displayed in Definition \ref{d: contract}.  Add a new column, a new row $v_s$, and modify the primed vectors so that together with $v_s$ they meet the two columns in one of the three patterns displayed just before Definition \ref{d: contract}.  Moreover, we require that $v_s^2 = - 2$, and the support of the two columns meets only the rows involved here.

\end{defin}

Thus, we may rephrase the preceding deductions as follows.

\begin{lem}\label{l: expand}

The set $\B_0$ consists of those matrices obtained from one of the matrices $M_1, M_2, M_3$ displayed in (\ref{e: possibility 1 & 2}) and (\ref{e: possibility 3}) by means of a sequence of expansions of the first and third type, up to permutations of its rows and columns.

\end{lem}

We abuse notation by identifying a row in a matrix with the corresponding row after an expansion.  In particular, every matrix $B \in \B_0$ has a distinguished pair of rows $v_i$ and $v_j$.


\subsection{Further analysis.}\label{ss: B_0 2}

Having characterized the matrices in the set $\B_0$, we now obtain finer information about a matrix $B \in \B_0$ for which $\langle v_i, v_j \rangle = 0$.

\begin{defin}\label{d: dist entry}

A {\em distinguished entry} in a matrix is one which is the only non-zero value in its column.

\end{defin}

Evidently, the process of expansion preserves the number of distinguished entries in a matrix.  Since the matrices appearing in (\ref{e: possibility 1 & 2}) and (\ref{e: possibility 3}) have two apiece, so does every $B \in \B_0$.

\begin{lem}\label{l: dist entry}

If two distinct rows $v_t$ and $v_u$ of $B \in \B_0$ contain a distinguished entry, then either $B = M_2$, or one of those rows has square $-2$.

\end{lem}

\begin{proof}

Suppose that neither row has square $-2$, and delete the columns containing the distinguished entries from $B$.  The result is an $r \times (r-1)$ matrix $\widehat{B}$, hence of rank $\leq r-1$.  On the other hand, $\widehat{B} \widehat{B}^T = \widehat{G} \oplus (-2)$ for a suitable matrix $\widehat{G}$.  Since this is an $r \times r$ matrix of rank $\leq r-1$, it follows that $\det(\widehat{G}) = 0$.  Let $\Gamma$ denote the graph with Goeritz matrix $G$.  By hypothesis, each of the vertices corresponding to $v_t$ and $v_u$ has an edge to the hub vertex $v_{r+1}$.  Removing these two edges results in a graph $\widehat{\Gamma}$ with Goeritz matrix $\widehat{G}$.  By the Matrix-Tree Theorem \cite[Theorem 5.6.8]{Stanley}, the quantity $|\det(\widehat{G})|$ equals the number of spanning trees of $\widehat{\Gamma}$.  Since this value is zero, it must be that $v_{r+1}$ is an isolated vertex in $\widehat{\Gamma}$.  This in turn implies that $r = \sum a_i = 2$, and we see at once that $B = M_2$, as claimed.

\end{proof}

\begin{lem}\label{l: expand 2}

If $B \in \B_0$ and $\langle v_i, v_j \rangle = 0$, then $B$ is constructed from one of $M_1$ or $M_2$ by a sequence of expansions of the first type.

\end{lem}

\begin{proof}

Given $B \in \B_0$, perform a sequence of contractions of the first type until no more are possible.  By Lemmas \ref{l: expand} and \ref{l: dist entry}, we obtain a matrix $M$ which is either one of the two matrices $M_1, M_2$ appearing in (\ref{e: possibility 1 & 2}), or distinct from these two and for which there is just one vector containing a distinguished entry.

Let us pursue this second possibility.  Perform a sequence of contractions of the third type to $M$ until no more are possible, resulting in a matrix $M'$.  Is it possible to perform a contraction of the first type to $M'$?  No, because there is still just one row containing a distinguished entry, and since it contains two such and some other non-zero entry, it has square $< -2$.  Therefore, there are no vectors of square $-2$ in $M'$ besides possibly one of $v_i$ and $v_j$.  However, neither $v_i$ nor $v_j$ can have square $-2$, for then we would have $M' = M_2$, and it is impossible to apply an expansion of the third type to it, so $M = M' = M_2$, in contradiction to the assumption $M \ne M_2$.  Hence, $M$ is the matrix $M_3$ of (\ref{e: possibility 3}).  Consider the sequence of expansions that carries the matrix $M_3$ to the original one $B$.  At no point does either vector $v_i$ or $v_j$ contain a $-1$ amongst the columns apt for expansion, nor does either contain a distinguished entry.  It follows that at no point does their inner product change from what they are in $M_3$; thus the rows $v_i$ and $v_j$ in $B$ have inner product $1$.

Therefore, if $\langle v_i, v_j \rangle = 0$, then the first possibility must occur, which proves the statement of the Lemma.


\end{proof}



\begin{lem}\label{l: C structure}

Given a matrix $B \in \B_0$ with $\langle v_i, v_j \rangle = 0$, its upper-right $r \times r$ submatrix can be put in the form
\begin{equation}\label{e: C}
C =
\left(
\begin{array}{cccccccccccc}
1 & \vline & -1     &   &   & \vline & & \vline & & & \\
 & \vline & & \ddots & & \vline & & \vline & & &  \\
 & \vline & *   & & -1     & \vline & & \vline & & &  \\
\hline
 & \vline & & & & \vline & 1 & \vline & -1    &   &    \\
 & \vline & & & & \vline & & \vline &  & \ddots &   \\
 & \vline & & & & \vline & & \vline & *   & & -1   \\
\hline
 & \vline & * & \cdots & * & \vline & & \vline & * & \cdots & *   \\
 & \vline & * & \cdots & * & \vline & & \vline & * & \cdots & *   \\
\end{array}
\right)
\begin{array}{c}
\\ \\ \\ \\ \\ \\ \\ v_i \\ v_j
\end{array}
\end{equation}
after reordering its rows and columns.  Here a starred entry takes the value 0 or 1, a blank one takes the value 0, and the truncations of rows $v_i$ and $v_j$ are labeled.  Moreover, in each of the four blocks of starred entries appearing in $v_i$ and $v_j$, at least one entry is non-zero.

\end{lem} \noindent For example, the block in the (1,2) position is a lower triangular matrix with $-1$'s on its diagonal and $0$'s and $1$'s below it.

\begin{proof}

Lemma \ref{l: expand 2} asserts that the submatrix $C$ can be obtained by applying a sequence of expansions of the first type to one of \[ \left( \begin{matrix} 1 & 1 \\ 0 & 0 \end{matrix} \right) \begin{array}{c} v_i \\ v_j \end{array} \quad \text{or} \quad \left( \begin{matrix} 1 & 0 \\ 0 & 1 \end{matrix} \right) \begin{array}{c} v_i \\ v_j. \end{array}\]  Let us examine how this can occur.  There is essentially one expansion that we can apply to either of these matrices, and after reordering its rows and columns, the resulting matrix takes the form \[ M_4 = \left( \begin{matrix} 1 & -1 & 0 \\ 0 & 1 & 0 \\ 0 & 1 & 1 \end{matrix} \right) \begin{array}{c} \\ v_i \\ v_j \end{array} \] in either case.  Now, a sequence of expansions will not change the inner product $\langle v_i, v_j \rangle = 1$ until we perform one for which $v_j$ plays the role of $a'$.  Once we perform such an expansion, neither of the resulting vectors $v_i$ or $v_j$ will contain a distinguished entry, so their inner product will remain constant throughout all subsequent expansions.  It follows that in order to obtain a matrix $B$ for which $\langle v_i, v_j \rangle = 0$, the expansion for which $v_j$ plays the role of $a'$ must use $v_i$ in the role of $b'$.  Furthermore, we can perform this expansion at the outset to the matrix $M_4$ to obtain \[ M_5 = \left( \begin{matrix} 1 & \vline & -1 & \vline & 0 & \vline & 0 \\ \hline 0 & \vline & 0 & \vline & 1 & \vline & -1 \\ \hline 0 & \vline & 1 & \vline & 0 & \vline & 1 \\ 0 & \vline & 1 & \vline & 0 & \vline & 1 \end{matrix} \right) \begin{array}{c} \\ \\ v_i \\ v_j, \end{array} \] and perform the remaining expansions in order, and thereby produce the same matrix $C$, up to reordering its rows and columns, as did the original sequence of expansions.  Similarly, we can reorder the remaining expansions, without affecting the resulting matrix $C$ up to its order of rows and columns, in the following way.  The first $k$ expansions have the property that the first row of $M_5$ plays the role of $a'$ in the first expansion, and in each of the next $k-1$, the role of $a'$ is played by the vector $v_s$ from the previous expansion.  Then the remaining $\ell$ expansions have the property that the second row of $M_5$ plays the role of $a'$ in the first such, and in each subsequent one, the role of $a'$ is played by the vector $v_s$ from the previous expansion.  To annotate this process, we append each new row/column pair created by one of the first $k$ expansions to the top/front of the matrix, and each new pair created by one the succeeding $\ell$ expansions to where the first horizontal line/second vertical line appear.  With this order of its rows and columns, the matrix $C$ takes the stated form.

\end{proof}

\subsection{Final\'e.}\label{ss: killer}

At last, suppose that $A$ is a matrix fulfilling the conclusions of Theorem \ref{t: criterion} for an alternating 3-braid knot $K$ with signature $0$, and suppose by way of contradiction that $\langle v_i, v_j \rangle = 0$.  Thus, its upper-right $r \times r$ submatrix $C$ takes the form described by Lemma \ref{l: C structure} after permuting its rows and columns.  By adding non-negative integer multiples of columns $2$ through $k+1$ to the first, and columns $k+3$ through $k + \ell + 2$ to the $(k+2)^{nd}$, we transform $C$ into a matrix $C'$ of equal determinant and of the form

\begin{equation}\label{e: C'}
C' =
\left(
\begin{array}{cccccccccccc}
& \vline & -1     &   &   & \vline & & \vline &  &  &\\
& \vline &  & \ddots & & \vline & & \vline &  &  & \\
& \vline & *   & & -1     & \vline & & \vline &  &  & \\
\hline
& \vline & & & & \vline & & \vline & -1    &   &   \\
& \vline & & & & \vline & & \vline &  & \ddots & \\
& \vline & & & & \vline & & \vline & *   & & -1   \\
\hline
\alpha & \vline & * & \cdots & * & \vline & \beta & \vline & * & \cdots & *  \\
\gamma & \vline & * & \cdots & * & \vline & \delta & \vline & * & \cdots & * \\
\end{array}
\right),
\end{equation} with each of $\alpha, \beta, \gamma, \delta \geq 1$.  Thus, Theorem \ref{t: criterion} implies that \begin{equation}\label{e: det 1} \pm 1 = \det C = \det C' = \pm \det  \left( \begin{matrix} \alpha & \beta \\ \gamma & \delta \end{matrix} \right). \end{equation}

Consider now the penultimate row $x$ of the matrix $A$, with its entries permuted in accordance with the permutation of the columns of $A$ that puts $C$ in the stated form (\ref{e: C}).  Its truncation to its last $r$ entries is a vector $\overline{x}$ which is orthogonal to the first $r-2$ rows of $C$.  By Lemma \ref{l: C structure}, it follows that this vector takes the form \[ \overline{x} = (t, m_1 t, \dots, m_k t, u, n_1 u, \dots, n_\ell u) \] for some integers $t,u$ and {\em positive} integers $m_1,\dots,m_k,n_1,\dots,n_\ell$.  The sequence of column operations that transforms $C$ into $C'$ carries $\overline{x}$ to a corresponding  vector $\overline{x}' = (t,*,\dots,*,u,*,\dots,*)$ for which $C \overline{x} = C' \overline{x}'$, which together with (\ref{e: C}) shows that in fact $\overline{x}' = (t,0,\dots,0,u,0,\dots,0)$. Since $x$ is orthogonal to the first $r$ rows of $A$, we obtain $C \overline{x} = (0,\dots,0,1,-1)^T$. It follows that \[ \left( \begin{matrix} \alpha & \beta \\ \gamma & \delta \end{matrix} \right) \left( \begin{matrix} t \\ u \end{matrix} \right) = \pm \left( \begin{matrix} 1 \\ -1 \end{matrix} \right),\] whence, by (\ref{e: det 1}), \[ \left( \begin{matrix} t \\ u \end{matrix} \right) = \pm \left( \begin{matrix} \delta & -\beta \\ -\gamma & \alpha \end{matrix} \right) \left( \begin{matrix} 1 \\ -1 \end{matrix} \right) = \pm \left( \begin{matrix} \delta + \beta \\ -(\gamma + \alpha) \end{matrix} \right).\]  However, in order for the vector $x$ to obey the change-making condition (\ref{e: change}), we must have \[ \min \{ |t|, |m_1 t|, \dots, |m_k t|, |u|, |n_1 u|, \dots, |n_\ell u| \}| \leq 1,\] while the left-hand side reduces to $\min \{ |t|, |u| \} = \min \{ \delta + \beta, \gamma + \alpha \} \geq 2$.  Hence $x$ fails the change-making condition.  It follows that no there is no matrix $A$ for which $\langle v_i, v_j \rangle = 0$ and which fulfills the conclusions of Theorem \ref{t: criterion} for an alternating 3-braid knot $K$ with $\sigma(K) = 0$.  This establishes Claim \ref{claim: (v_i,v_j) = 1}, and completes the proof of Theorem \ref{t: criterion}.


\section{Conclusion.}\label{s: conc}

\subsection{Non-alternating 3-braid knots.}

A slight modification of Theorem \ref{t: criterion} applies to any knot $K$ whose branched double-cover is an L-space which bounds a sharp 4-manifold $X$.  The relevant change is that the Goeritz matrix $G_K$ must be replaced by a matrix representing the intersection pairing $Q_X$.  In particular, it applies to at least one of $K$ and $\overline{K}$ for a 3-braid knot $K$ for which $d= \pm 1$ when put in normal form, although we do not elaborate on the construction of $X$ here.  Granted this fact, we can try to proceed exactly as with the case of an alternating 3-braid knot.  The story begins to unfold as in Section \ref{s: sig 0}, and we confront a combinatorial problem analogous to describing the set $\B_0$.  However, things quickly grow complicated.  For example, the analogous building blocks to $M_1,M_2,M_3$ are much more numerous, and no simple way for enumerating and handling them all emerged to this author.  In principle, this is a tractable problem, and with enough effort we could hope to prove the following result.

\begin{conj}\label{conj: 3braids}

Suppose that $K$ is a 3-braid knot with unknotting number one.  Then $K$ contains an unknotting crossing in normal form, and $|d| \leq 1$.

\end{conj}

We note that there certainly do exist 3-braid knots with unknotting number one and $d = \pm 1$, such as $8_{20}$ and $8_{21}$.  Comparing with Proposition \ref{p: d bound}, Conjecture \ref{conj: 3braids} asserts that $d = 2$ is not a possibility.  Let us provide some justification for this assertion.  Suppose that $K$ is the closure of $h^2 \cdot w$ with $w$ an alternating word, and let $K_0$ denote the closure of $w$.  The space $Y = \Sigma(K)$ is obtained by $(-1)$-surgery on a suitable null-homologous knot in $\Sigma(K_0)$.  We perform the corresponding handle attachment to $\Sigma(K_0) = \del X_{K_0}$, and in this way produce a negative definite 4-manifold with boundary $Y$.  While $Y$ is not an L-space, it is nearly so, in that $\rk \; HF_{red}(Y) = 1$, and its correction terms can be determined from those of $\Sigma(K_0)$ \cite[Theorem 6.2]{B}.  The argument of \cite[Section 4]{OSunknotting} then pushes through in this case to give an analogous statement to Theorem \ref{t: symmetry}.  Thus, we can try to proceed in this case, just as with $|d| \leq 1$.  However, as with the case $d = \pm 1$ just discussed, an avalanche of case analysis halted this author's progress.  Nevertheless, we expect this analysis to show that there is no 3-braid knot with $d=2$ and unknotting number one.  As an exercise, the reader may try to argue that the normal form of such a knot cannot contain an unknotting crossing.  In this vein, we pose a question.  According to \cite[Proposition 1.6]{B}, a 3-braid knot $K$ with 4-ball genus $g_4(K)=0$ must have $|d| \leq 1$ (which follows, alternatively, by combining the calculations of $\sigma$ and $s$ in Propositions \ref{p: sig} and \ref{p: s}).

\begin{q}

If $K$ is a 3-braid knot and $g_4(K)=1$, does it follow that $|d| \leq 1$?

\end{q}  \noindent If so, this would directly yield the last assertion of Conjecture \ref{conj: 3braids}.

\subsection{Quasi-alternating links and sharp 4-manifolds.}

Let $Y$ denote the branched double-cover of a quasi-alternating link.  This space bounds a negative definite 4-manifold with vanishing $H_1$, and one might be inclined to believe that it necessarily bounds a sharp 4-manifold.  However, Lemma \ref{l: d diff 2} can be used to show that this is not always the case.

\begin{prop}\label{p: non-qa}

The branched double-cover of the knot $\overline{8}_{20}$ does not bound a sharp 4-manifold.

\end{prop}

\noindent  Proposition \ref{p: non-qa} is a negative result, and begs for an effecient means of calculating the correction terms of the branched double-cover of a quasi-alternating link in general.  We remark that the space in question is the result of $(-9)$-surgery on the right-hand trefoil, whose associated trace of surgery is a negative definite 4-manifold with vanishing $H_1$.

\begin{proof}

\begin{figure}
\centering
\includegraphics[width=2.5in]{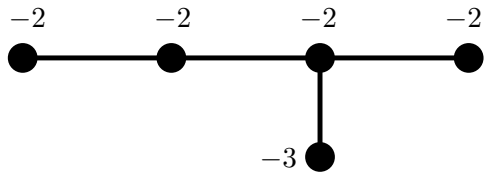}
\put(-180,55){$-2$}
\put(-125,55){$-2$}
\put(-70,55){$-2$}
\put(-15,55){$-2$}
\put(-85,2){$-3$}
\caption{Plumbing along this graph gives a sharp 4-manifold with boundary $\Sigma(8_{20})$.}  \label{f: plumbing}
\end{figure}

The knot $8_{20}$ is the pretzel knot $P(3,-3,2)$, so the space $Y = \Sigma(8_{20})$ is the boundary of plumbing on the graph shown in Figure \ref{f: plumbing}.  This is a sharp 4-manifold according to \cite[Theorem 1.2]{OSplumbed}.  The associated intersection pairing is

\[
M =
\left(
\begin{matrix}
-2 &  1 &  0 &  0 &  0 \\
 1 & -2 &  1 &  0 &  0 \\
 0 &  1 & -2 &  1 &  1 \\
 0 &  0 &  1 & -2 &  0 \\
 0 &  0 &  1 &  0 & -3
\end{matrix}
\right),
\] which it is easy to check has a unique embedding, up to automorphism, into $-\Z^5$, and two embeddings into $-\Z^n$ for all $n > 5$; the embedding matrices are

\[
A_1 =
\left(
\begin{matrix}
 1 & -1 &  0 &  0 &  0 & 0 & \cdots & 0 \\
 0 &  1 & -1 &  0 &  0 & 0 & \cdots & 0 \\
 0 &  0 &  1 & -1 &  0 & 0 & \cdots & 0 \\
 0 &  0 &  0 &  1 & -1 & 0 & \cdots & 0 \\
-1 & -1 & -1 &  0 &  0 & 0 & \cdots & 0
\end{matrix}
\right)
\quad \text{and} \quad
A_2 =
\left(
\begin{matrix}
 1 & -1 &  0 &  0 &  0 & 0 & 0 & \cdots & 0 \\
 0 &  1 & -1 &  0 &  0 & 0 & 0 & \cdots & 0 \\
 0 &  0 &  1 & -1 &  0 & 0 & 0 & \cdots & 0 \\
 0 &  0 &  0 &  1 & -1 & 0 & 0 & \cdots & 0 \\
 0 &  0 &  0 &  1 &  1 & 1 & 0 & \cdots & 0
\end{matrix}
\right).
\]  Suppose by way of contradiction that $\Sigma(\overline{8_{20}}) = -Y = \del X$, with $X$ sharp.  By Lemma \ref{l: d diff 2}, it follows that there exists some $n \geq 5$ and value $i \in \{ 1,2 \}$ such that for every class $\t \in \cok(M)$, we can find a vector $\alpha \in \{ \pm 1 \}^n$ for which $[A_i \cdot \alpha] = \t$.  However, it is straightforward to check that there is no such $\alpha$ corresponding to $20$ of the $25$ classes $\t$ in case of $A_1$, and to $\t = 0$ in case of $A_2$.

\end{proof}

\subsection{Further speculation.}

If $K$ is an alternating knot with unknotting number one, then the knot $\kappa$ guaranteed by the Montesinos trick is a knot admitting an L-space surgery, or an {\em L-space knot}, for short.  L-space knots are very special.  For example, the Alexander polynomial and knot Floer homology of an L-space knot are highly constrained \cite{OSlens}, and the knot must be fibered \cite{Ghiggini,Juhasz,Ni}.  Furthermore, there is a conjecturally complete list, due to Berge, of the knots admitting a lens space surgery \cite{Berge}.  Inspired by this body of work, it seems plausible that a deeper understanding of the topology of L-space knots could ultimately lead to their classification.  According to this line of thought, the constraints on the knot $\kappa$ might be so strong that it must arise in correspondence to an unknotting crossing in an alternating diagram of $K$, and thereby establish Conjecture \ref{conj: main}.


At a more approachable level, we may ask to what extent Theorem \ref{t: criterion} captures the full strength of the symmetry stated in Theorem \ref{t: symmetry}.  As it stands, the proof of Theorem \ref{t: criterion} only makes use of Equation (\ref{e: symmetry}) in the event that the differences therein vanish.  This seems like a tractable combinatorial problem, and its resolution could be useful in understanding Question \ref{q: main}.


\begin{thebibliography}{99}

\bibitem{B}

	J.A. Baldwin.  Heegaard Floer homology and genus one, one boundary component open books. {\em J. Topology} 1 (4) (2008)  963-992.

\bibitem{Berge}

	J. Berge.  Some knots with surgeries yielding lens spaces.  {\em Unpublished manuscript}.

\bibitem{knotinfo}

	J. C. Cha and C. Livingston. KnotInfo: Table of Knot Invariants, {\tt http://www.indiana.edu/~knotinfo}, December 10, 2008.

\bibitem{CL}

	T. D. Cochran and W. B. R. Lickorish.  Unknotting information from 4-manifolds.  {\em Trans. Amer. Math. Soc.} 297 (1986) 125-142.

\bibitem{D}

	S.K. Donaldson.  The orientation of {Y}ang-{M}ills moduli spaces and 4-manifold topology. {\em J. Diff. Geom.} 26(3) (1987) 397-428.

\bibitem{E}

	D. Erle.  Calculation of the signature of a 3-braid link.  {\em Kobe J. Math.} 16(2) (1999) 161-175.

\bibitem{Ghiggini}

	P. Ghiggini.  Knot Floer homology detects genus-one fibred knots.  {\em Amer. J. Math.}  130 no. 5 (2008) 1151-1169.

\bibitem{GLi}

	C. McA. Gordon and R. Litherland.  On the signature of a link. {\em Invent. Math.} 47 (1978) 53-69.

\bibitem{GLu}

	C. McA. Gordon, J. Luecke.  Knots with unknotting number 1 and essential Conway spheres. {\em Alg. Geom. Top.}  6  (2006) 2051-2116.

\bibitem{GJ}

	J. Greene and S. Jabuka.  The slice-ribbon conjecture for 3-stranded pretzel knots. {\tt arXiv:0706.3398} (2007).

\bibitem{Jabuka}

	S. Jabuka. The Witt unknotting number of a knot. {\em To appear}.

\bibitem{Juhasz}

	A. Juh\'asz. Floer homology and surface decompositions.  {\em Geom. Top.}  12 no. 1 (2008) 299-350.

\bibitem{KM}

    T. Kanenobu and H. Murakami, 2-bridge knots of unknotting number one. {\em Proc. Amer. Math. Soc.} 96 no.3 (1986), 499-502.

\bibitem{Kauffman}

	L. Kauffman.  State models and the Jones polynomial.  {\em Topology} 26 (1987) 395-407.

\bibitem{Kobayashi}

    T. Kobayashi, Minimal genus Seifert surfaces for unknotting number 1 knots. {\em Kobe J. Math} 6 (1989) 53-62.

\bibitem{Kohn}

	P. Kohn.  Two-bridge links with unlinking number one.  {\em Proc. Amer. Math. Soc.} 113 no. 4 (1991) 1135-1147.

\bibitem{Lick}

	W.B.R. Lickorish.  \emph{An introduction to knot theory}, Graduate Texts in Math. 175, Springer, New York (1997).

\bibitem{Lisca}

	P. Lisca.  Lens spaces, rational balls and the ribbon conjecture. {\em Geom. Top.} 11 (2007) 429-472

\bibitem{Liv}

	C. Livingston.  Computations of the Ozsv\'ath-Szab\'o knot concordance invariant.  {\em Geom. Top.} 8 (2004) 735-742.

\bibitem{MOw}

	C. Manolescu and B. Owens.  A concordance invariant from the Floer homology of double branched covers. {\em Int. Math. Res. Not.} {\tt doi:10.1093/imrn/rnm077} (2007).

\bibitem{MOzs}

	C. Manolescu and P. Ozsv\'ath.  On the Khovanov and knot Floer homologies of quasi-alternating links. {\em Proc. Fourteenth G\"{o}kova Geom. Top. Conf.} (2007) 60-81.

\bibitem{MenThis}

	W. Menasco and M. Thistlethwaite.  The Tait flyping conjecture.  {\em Bull. Amer. Math. Soc.}  25 (1991) 403-412.

\bibitem{Monty}

    J.M. Montesinos. Surgery on links and double branched covers of $S^3$. in: Knots, Groups, and 3-manifolds (Papers dedicated to the memory of R. H. Fox), {\em Ann. Math. Studies}, No. 84 (1975) 227–259.

\bibitem{Msig}

	K. Murasugi.  On a certain numerical invariant of link types.  {\em Trans. Amer. Math. Soc.} 117 (1965) 387-422.

\bibitem{M3braid}

	\underline{\hspace{.5in}}. {\em On closed 3-braids.}  Mem. Amer. Math. Soc. 151 (1974).

\bibitem{Malt}

	\underline{\hspace{.5in}}. Jones polynomials and classical conjectures in knot theory.  {\em Topology}  26  (1987) 187-194.

\bibitem{Ni}

	Y. Ni. Knot Floer homology detects fibred knots. {\em Invent. Math.}  170 no. 3 (2007) 577-608.

\bibitem{OSgrading}

	P. Ozsv\'ath and Z. Szab\'o.  Absolutely graded Floer homologies and intersection forms for four-manifolds with boundary. {\em Adv. Math.} 173 (2003) 179-261.

\bibitem{OStau}

	\underline{\hspace{.5in}}.  Knot Floer homology and the four-ball genus. \emph{ Geom. Top.} 7 (2003) 615-639.

\bibitem{OSplumbed}

	\underline{\hspace{.5in}}.  On the Floer homology of plumbed three-manifolds.  \emph{Geom. Top.} 7 (2003) 185-224.

\bibitem{OSlens}

    \underline{\hspace{.5in}}.  On knot Floer homology and lens space surgeries.  {\em Topology} 44 (2005) 1281-1300.

\bibitem{OSdoublecover}

    \underline{\hspace{.5in}}.  On the Heegaard Floer homology of branched double-covers.  \emph{Adv. Math.} 194 (2005) 1-33.

\bibitem{OSunknotting}

	\underline{\hspace{.5in}}.  Knots with unknotting number one and Heegaard Floer homology. {\em Topology} 44 (2005) 705-745.

\bibitem{R}

	J. Rasmussen.  Khovanov homology and the slice genus.  {\tt math.GT/0402131} (2004).

\bibitem{Staron}

    E. Staron.  {\em Private communication.} (February 2009).

\bibitem{Stanley}

    R. Stanley. {\em Enumerative Combinatorics, Volume 2.} Cambridge Studies in Adv. Math. 62, Cambridge Univ. Press (1999).

\bibitem{Stoim}

	A. Stoimenow.  Some examples related to 4-genera, unknotting numbers and knot polynomials.  {\em J. London Math. Soc.} (2) 63 (2001) 487-500.

\bibitem{Thistle}

	M. B. Thistlethwaite.  A spanning tree expansion of the Jones polynomial.  {\em Topology}  26  (1987) 297-309.

\bibitem{T}

	T. Tsukamoto.  The almost alternating diagrams of the trivial knot. {\tt math.GT/0605018} (2006).

\bibitem{V}

	C. A. Van Cott.  Ozsv\'ath-Szab\'o and Rasmussen invariants of cabled knots. {\tt arXiv:0803.0500} (2008).

\end{thebibliography}
\end{document}